\documentclass[a4paper,11pt,reqno]{amsart}
\usepackage[utf8]{inputenc}
\usepackage{amsmath}
\usepackage{amssymb}
\usepackage{amsthm}
\usepackage{hyperref}
\usepackage{color}
\usepackage{a4wide}
\usepackage{dsfont}
\usepackage{mathdots}
\usepackage{tikz-cd}
\makeatletter

\newcommand{\fraka}{\mathfrak{a}}

\newcommand{\frakn}{\mathfrak{n}}

\newcommand{\CC}{\mathbb{C}}

\newcommand{\NN}{\mathbb{N}}

\newcommand{\RR}{\mathbb{R}}
\newcommand{\ZZ}{\mathbb{Z}}

\newcommand{\calD}{\mathcal{D}}

\newcommand{\calO}{\mathcal{O}}

\newcommand{\R}{\mathbb{R}}
\newcommand{\C}{\mathbb{C}}

\DeclareMathOperator{\GL}{GL}
\DeclareMathOperator{\SL}{SL}
\DeclareMathOperator{\upO}{O}

\DeclareMathOperator{\upU}{U}

\DeclareMathOperator{\Ind}{Ind}
\DeclareMathOperator{\tr}{tr}

\DeclareMathOperator{\ad}{ad}

\DeclareMathOperator{\Hom}{Hom}

\DeclareMathOperator{\rest}{rest}

\DeclareMathOperator{\id}{id}
\DeclareMathOperator{\sgn}{sgn}
\DeclareMathOperator{\const}{const}
\DeclareMathOperator{\diag}{diag}

\renewcommand\Re{\operatorname{Re}}
\renewcommand\Im{\operatorname{Im}}

\newcommand{\ps}{\textup{p.s.}}

\theoremstyle{plain}
\newtheorem{theorem}{Theorem}[section]
\newtheorem{proposition}[theorem]{Proposition}
\newtheorem{lemma}[theorem]{Lemma}
\newtheorem{corollary}[theorem]{Corollary}

\newtheorem{thmalph}{Theorem}

\theoremstyle{definition}

\newtheorem{remark}[theorem]{Remark}

\title[Branching laws for Stein's complementary series and Speh representations]{Branching laws for Stein's complementary series and Speh representations of $\mathrm{GL}(2n,\RR)$}

\author{Jonathan Ditlevsen}
\address{Graduate School of Mathematical Sciences, The University of Tokyo, 3-8-1 Komaba, Meguro-ku, Tokyo 153-8914, Japan}
\email{JonathanDitlevsen@gmail.com}
\author{Jan Frahm}
\address{Department of Mathematics, Aarhus University, Ny Munkegade 118, 8000 Aarhus C, Denmark}
\email{frahm@math.au.dk}

\begin{document}

\begin{abstract}
    We obtain the explicit direct integral decomposition of Stein's complementary series representations and Speh representations of $\GL(2n,\RR)$ when restricted to the subgroup $\GL(2n-1,\RR)$. The decomposition is a direct integral of unitarily induced representations from a maximal parabolic subgroup of $\GL(2n-1,\RR)$ with Levi factor $\GL(2n-2,\RR)\times\GL(1,\RR)$, where the induction data consists of a complementary series or Speh representation of the factor $\GL(2n-2,\RR)$ with the same parameter as the one of $\GL(2n,\RR)$ and a character of $\GL(1,\RR)$. These results are in line with the theory of adduced representations.\\
    The main tools in the proof are two families of symmetry breaking operators between degenerate series representations of $\GL(2n,\RR)$ and $\GL(2n-1,\RR)$ whose meromorphic properties are studied in great detail.
\end{abstract}

\maketitle

\section*{Introduction}

In the representation theory of the general linear group over the real numbers, there are two classes of representations that serve as building blocks for all irreducible unitary representations: Speh representations and complementary series representations (see \cite{Vog86} for details on this viewpoint). The simplest class of complementary series representations is the real counterpart of the one for complex groups discovered by Stein~\cite{Ste67}. Both Stein's complementary series and Speh representations only exist in even dimensions, i.e. for $G=\GL(2n,\RR)$. 

In this paper, we study the problem of decomposing their restriction to the subgroup $H=\GL(2n-1,\RR)$ into a direct integral of irreducible unitary representations. We show that they decompose into a direct integral of unitarily induced representations from a maximal parabolic subgroup of $H$ with Levi factor $\GL(2n-2,\RR)\times\GL(1,\RR)$, where the induction data consists of a representation of $\GL(2n-2,\RR)$ of the same type and parameter as the one of $\GL(2n,\RR)$ and a character of $\GL(1,\RR)$.

For this, we view both Stein's complementary series and Speh representations as sitting inside the same family of degenerate series representations, a family of non-unitarily induced representations from a maximal parabolic subgroup of $G$. This family is essentially parameterized by a complex number and contains four different classes of irreducible unitarizable representations:
\begin{itemize}
    \item Unitary degenerate series representations,
    \item Stein's complementary series representations,
    \item Small irreducible quotients and
    \item Speh representations.
\end{itemize}

Our approach to the decomposition problem is to first decompose the restriction to $H$ of all unitary degenerate series representations. This is reasonably straightforward since the unitary degenerate series consists of unitarily induced representations, so we can apply the Mackey machinery. The resulting decomposition can be made explicit in the sense of a Plancherel formula, i.e. a decomposition of vectors in the representation space into their irreducible components with respect to $H$. The Plancherel formula involves a family of intertwining operators from the degenerate series of $G$ to a certain degenerate series of $H$ induced from a non-maximal parabolic subgroup, so-called symmetry breaking operators, a term coined by Kobayashi~\cite{Kob15}. The main body of this paper consists of a detailed study of these operators, in particular their meromorphic dependence on the representations parameters, resulting in an analytic continuation of the Plancherel formula to those parameters that correspond to complementary series representations or irreducible quotients such as the Speh representations. This technique has previously been applied in other settings, see \cite{Wei24} for $(G,H)=(\upO(1,n+1),\upO(1,n))$ and \cite{FW} for $(G,H)=(\upU(1,n+1),\upU(1,n))$, and there is more work in progress for other cases such as $(G,H)=(\GL(3,\RR),\GL(2,\RR))$ or $(\upO(p+1,q+1),\upO(p+1,q))$.

We remark that the direct integral decompositions obtained in this paper can also be deduced from the theory of adduced representations (see \cite{Sah89}) and the computation of adduced representations for Stein's complementary series by Sahi~\cite{Sah90} and for Speh representations by Sahi--Stein~\cite{SS90} (see Remark~\ref{rem:Adduced} for details). Our proof differs from this alternative approach in the way that we simultaneously decompose \emph{all} irreducible unitary representations that occur in the degenerate series of $G$ by an analytic continuation procedure. In particular, we also obtain the direct integral decomposition for the small irreducible quotients without additional work. Moreover, our decomposition is explicit in the sense that we are able to decompose vectors in the representation spaces using symmetry breaking operators. The analysis of families of symmetry breaking operators has in the past decade developed into a research topic on its own (see \cite{Kob15} for an overview), and this paper can also be viewed as a contribution to this study.

We now describe our results in some detail.

\subsection*{Branching laws for unitary representations}

Let $G=\GL(2n,\RR)$ and consider the standard parabolic subgroup $P_G\subseteq G$ with Levi factor $\GL(n,\RR)\times \GL(n,\RR)$. We form the degenerate series representations 
\[
\pi_{\xi,\lambda}=\Ind_{P_G}^G(\xi\otimes e^\lambda\otimes 1)
\]
induced from a character of $P_G$ which is parameterized by $\xi\in (\ZZ/2\ZZ)^2$ and $\lambda\in \CC^2$ (see Section~\ref{sec:DegSerG} for details). For $\lambda\in i\RR^2$, the representations $\pi_{\xi,\lambda}$ are unitary and irreducible on the space of $L^2$-sections of the corresponding line bundle over $G/P_G$. Recall that between the principal series representations we have a holomorphic family of Knapp--Stein intertwining operators (see Section~\ref{sec:KSforG})
\[
\mathbf{T}_{\xi,\lambda}:\pi_{\xi,\lambda}\to \pi_{w_G(\xi,\lambda)},
\]
where $w_G$ is the unique non-trivial element of the Weyl group associated with $P_G$. We can use these intertwining operators to produce an invariant Hermitian form on $\pi_{\xi,\lambda}$ in the case where $\xi_1=\xi_2$ and $\lambda_1=-\lambda_2\in\RR$, and when this form is positive definite resp. semidefine the degenerate series $\pi_{\xi,\lambda}$ resp. its quotient $\pi_{\xi,\lambda}/\ker(\mathbf{T}_{\xi,\lambda})$ is irreducible and unitarizable. This yields the following list of irreducible unitary representations of $G$ for $\xi_1=\xi_2$ (see Section~\ref{sec:UnirrepsG} for details):
\begin{enumerate}
    \item \emph{Unitary degenerate series} $\pi_{\xi,\lambda}^{\mathrm{unit.}}=\pi_{\xi,\lambda}$ for $\lambda\in i\RR^2$,
    \item \emph{Stein's complementary series} $\pi_{\xi,\lambda}^{\operatorname{c.s.}}=\pi_{\xi,\lambda}$ for $\lambda_1=-\lambda_2\in(-\frac{1}{2},\frac{1}{2})\setminus\{0\}$,
    \item \emph{Small irreducible quotients} $\pi_{\xi,\lambda}^{\operatorname{small}}=\pi_{\xi,\lambda}/\ker(\mathbf{T}_{\xi,\lambda})$ for $\lambda_1=-\lambda_2\in \{\tfrac{1}{2},1,\dots,\frac{n}{2}\}$,
    \item \emph{Speh representations} $\pi_{\xi,\lambda}^{\operatorname{Speh}}=\pi_{\xi,\lambda}/\ker(\mathbf{T}_{\xi,\lambda})$ for $\lambda_1=-\lambda_2\in \{-\frac{1}{2},-\frac{3}{2},-\frac{5}{2},\ldots\}$.
\end{enumerate}

Our main result is the explicit decomposition of the restriction of these representations to the subgroup $H=\GL(2n-1,\RR)$ into a direct integral of irreducible unitary representations. To describe the representations of $H$ that appear in this decomposition, we consider the maximal parabolic subgroup $Q_H$ of $H$ with Levi factor $\GL(2n-2)\times \GL(1,\RR)$. In the same way as above for $\GL(2n,\RR)$, we denote by $\varpi_{\xi,\lambda}^{\mathrm{unit.}}$, $\varpi_{\xi,\lambda}^{\textrm{c.s.}}$, $\varpi_{\xi,\lambda}^{\textrm{small}}$ and $\varpi_{\xi,\lambda}^{\textrm{Speh}}$ the unitary degenerate series, Stein's complementary series, small irreducible quotients and Speh representations of $\GL(2n-2,\RR)$, respectively. Furthermore, we write $\chi_{\eta,\nu}$ for the character $x\mapsto |x|^\nu_\eta=\sgn(x)^\eta|x|^\nu$ ($\eta\in \ZZ/2\ZZ$, $\nu\in \CC$) of $\GL(1,\RR)=\RR^\times$. For $\varpi$ an irreducible unitary representation of $\GL(2n-2,\RR)$ and $\nu\in i\RR$ the parabolically induced representation $\Ind_{Q_H}^H(\varpi\otimes\chi_{\eta,\nu})$ is irreducible and unitary (see Section~\ref{sec:UnirrepsH}).

Our main result is the following decomposition: 

\begin{thmalph}[see Theorem~\ref{thm:BranchingLaws} and Corollary~\ref{cor:BranchingLaws}]\label{thm:BranchingLawsIntro}
For $\square\in \{\operatorname{unit.},\operatorname{c.s.},\operatorname{small},\operatorname{Speh}\}$:
\[
\pi_{\xi,\lambda}^\square|_H\simeq \bigoplus_{\eta\in\ZZ/2\ZZ}\int_{i\RR_+}^\oplus \Ind_{Q_H}^H(\varpi_{\xi,\lambda}^\square\otimes\chi_{\eta,\nu})\,d\nu.
\]
\end{thmalph}

Our strategy of proof is to first decompose unitary degenerate series representations for $\lambda\in i\RR^2$ explicitly using Mackey theory and then extending this decomposition analytically in the parameter $\lambda\in\CC^2$ to decompose the other three families of representations. Let us briefly describe this procedure.

For $\lambda\in i\RR^2$ the Hilbert space on which the unitary representation $\pi_{\xi,\lambda}$ is realized is a space of $L^2$-sections over $G/P_G$. The subgroup $H$ acts with an open dense orbit on $G/P_G$ and a straightforward application of the Mackey machinery (in the same spirit as e.g. in \cite{TOP11}) together with the Plancherel formula for the classical Mellin transform yields the following decomposition (see Corollary~\ref{cor:UnitaryPlancherel}):
\[
\pi_{\xi,\lambda}|_H\simeq \bigoplus_{\eta\in\ZZ/2\ZZ}\int_{i\RR_+} \tau_{(\xi_1,\eta,\xi_2),(\lambda_1,\nu,\lambda_2)}\,d\nu,
\]
where $\tau_{(\xi_1,\eta,\xi_2),(\lambda_1,\nu,\lambda_2)}$ is a certain degenerate series of $H$ induced from a character of the standard parabolic subgroup $P_H$ of $H$ with Levi factor $\GL(n-1,\RR)\times \GL(1,\RR)\times \GL(n-1,\RR)$ and parameterized by $(\xi_1,\eta,\xi_2)\in (\ZZ/2\ZZ)^3$ and $(\lambda_1,\nu,\lambda_2)\in \CC^3$ (see Section~\ref{subsec:Hreps}). To use analytic extension in $\lambda\in\CC^2$, we need to make this decomposition explicit on the level of vectors, so we first derive a Plancherel type formula
\[
\|f\|_{L^2}^2 =\sum_{\eta\in\ZZ/2\ZZ}\int_{i\RR}\|B_{\xi,\lambda}^{\eta,\nu}f\|^2_{L^2}\,d\nu \qquad (f\in\pi_{\xi,\lambda})
\]
given in terms of a certain family of symmetry breaking operators $B_{\xi,\lambda}^{\eta,\nu}\in\Hom_H(\pi_{\xi,\lambda}|_H,\allowbreak \tau_{(\xi_1,\eta,\xi_2),(\lambda_1,\nu,\lambda_2)})$. The majority of the paper is concerned with extending both sides of the formula meromorphically in $\lambda\in\CC^2$ to also cover those parameters for which $\pi_{\xi,\lambda}$ contains complementary series, small irreducible quotients or Speh representations (see Section~\ref{sec:AnalyticContinuation}). Finally, the degenerate series representations $\tau_{(\xi_1,\eta,\xi_2),(\lambda_1,\nu,\lambda_2)}$ and their irreducible quotients are related to the representations $\Ind_{Q_H}^H(\varpi\otimes\chi_{\eta,\nu})$ via standard intertwining operators (see Theorem~\ref{thm:SmallInvariantHermitianFormPositiveSemidefinite}), which proves Theorem~\ref{thm:BranchingLawsIntro}.

\subsection*{Symmetry breaking operators}

The symmetry breaking operators $B_{\xi,\lambda}^{\eta,\nu}$ are given by the integral
\[
B_{\xi,\lambda}^{\eta,\nu}f(h)=\int_{\RR^\times} |t|^{\lambda_1-\nu-\frac{n}{2}}_{\xi_1+\eta} f(he^{tE_{2n,n}})\frac{dt}{|t|},
\]
whenever it converges (see Section~\ref{subsec:A-SBO}). Here, $E_{2n,n}$ denotes the $2n\times 2n$ matrix with a one in the entry $(2n,n)$ and zeros elsewhere. We find a meromorphic function $n_\mathbf{B}(\xi,\lambda,\eta,\nu)$ so that the normalized family $\mathbf{B}_{\xi,\lambda}^{\eta,\nu}=n_\mathbf{B}(\xi,\lambda,\eta,\nu)^{-1}B_{\xi,\lambda}^{\eta,\nu}$ extends holomorpically to $(\lambda,\nu)\in \CC^2\times \CC$ for all $(\xi,\eta)\in (\ZZ/2\ZZ)^2\times (\ZZ/2\ZZ)$ (see Proposition~\ref{prop:B-SBOholomorphic}).

Since the invariant Hermitian form that induces the invariant inner product on the complementary series and the unitarizable quotients of $\pi_{\xi,\lambda}$ involves the intertwining operator  $\mathbf{T}_{\xi,\lambda}$, we are led to consider the composition $\mathbf{B}_{w_G (\xi ,\lambda )}^{\eta ,\nu }\circ \mathbf{T} _{\xi ,\lambda }$ which turns out to be another type of symmetry breaking operator, namely (up to a constant) the operator $A_{\xi,\lambda}^{\eta,\nu}\in \Hom_H(\pi_{\xi,\lambda}|_H,\allowbreak \tau_{(\xi_2,\eta,\xi_1),(\lambda_2,\nu,\lambda_1)})$ given by the integral
\[
A_{\xi,\lambda}^{\eta,\nu}f(h)=\int_{\RR^{n\times n}} |\textstyle{\det(X)}|^{\lambda_1-\nu-\frac{n}{2}}_{\xi_1+\eta}|\det_{n-1}(X)|^{\nu-\lambda_2-\frac{n}{2}}_{\eta+\xi_2}f(h\begin{pmatrix}
    I_n & \\
    X & I_n
\end{pmatrix})\,dX,
\]
whenever it converges (see Section~\ref{subsec:A-SBO}). Here, $\det_{n-1}$ denotes the determinant of the upper left $(n-1)\times(n-1)$ minor. Again we find a meromorphic function $n_{\mathbf{A}}(\xi,\lambda,\eta,\nu)$ such that $\mathbf{A}_{\xi,\lambda}^{\eta,\nu}=n_{\mathbf{A}}(\xi,\lambda,\eta,\nu)^{-1}A_{\xi,\lambda}^{\eta,\nu}$ extends holomorphically to $(\lambda,\nu)\in \CC^2\times \CC$ for all $(\xi,\eta)\in (\ZZ/2\ZZ)^2\times (\ZZ/2\ZZ)$ (see Proposition~\ref{prop:A-SBOholomorphic}).

One of the key statements in our analysis of symmetry breaking operators are the following \emph{functional equations} which relate the composition of symmetry breaking operators with the Knapp--Stein intertwining operators $\mathbf{T}_{\xi,\lambda}$ for $G$ and a similar family of intertwining operators $\mathbf{S}_{\eta,\nu}:\tau_{\eta,\nu}\to \tau_{w_H(\eta,\nu)}$ for $H$ (see Section~\ref{sec:KSforH} for their definition):

\begin{thmalph}[see Proposition~\ref{Prop: functional equations}]\label{thm:FunctionalEquationsIntro}
    We have the following relations between $\mathbf{A}_{\xi,\lambda}^{\eta,\nu}$ and $\mathbf{B}_{\xi,\lambda}^{\eta,\nu}$:
    \begin{equation*}
\mathbf{S}_{(\xi_1,\eta,\xi_2) ,(\lambda_1,\nu,\lambda_2) } \circ \mathbf{B}_{\xi ,\lambda }^{\eta ,\nu }=\frac{1}{n_\mathbf{B}(\xi,\lambda,\eta,\nu)} \mathbf{A}_{\xi ,\lambda }^{\eta ,\nu }
\end{equation*}
and
\begin{equation*}
\mathbf{B}_{w_G (\xi ,\lambda )}^{\eta ,\nu }\circ \mathbf{T} _{\xi ,\lambda }=\frac{\sqrt{\pi}(-1)^{(\eta+\xi_2)(\xi_1+\xi_2)}}{\Gamma(\frac{\lambda_2-\lambda_1+n+[\xi_2+\xi_1]}{2})} \mathbf{A}_{\xi ,\lambda }^{\eta ,\nu }.
\end{equation*}
\end{thmalph}

Further results for $\mathbf{A}_{\xi,\lambda}^{\eta,\nu}$ and $\mathbf{B}_{\xi,\lambda}^{\eta,\nu}$ include mapping properties, vanishing, estimates in the parameters and residue formulas (see Proposition~\ref{prop:Bvanishing}, Remark~\ref{rem:DiffOpResidues}, Proposition~\ref{prop:B-surjective}, Lemma~\ref{lem:VanishingOfQ} and Theorem~\ref{thm:DecayOfBinNu}).

\subsection*{Relation to other work}

As pointed out earlier, Theorem~\ref{thm:BranchingLawsIntro} for complementary series and Speh representations can also be deduced from the work of Sahi~\cite{Sah90} and Stein--Sahi~\cite{SS90} (see Remark~\ref{rem:Adduced} for details). Our approach is different and uses in a crucial way the families $\mathbf{A}_{\xi,\lambda}^{\eta,\nu}$ and $\mathbf{B}_{\xi,\lambda}^{\eta,\nu}$ of symmetry breaking operators. For real infinitesimal characters, the existence of such operators is a special case of the recent more general results by Kobayashi--Speh~\cite{KS25}. Our construction and analysis of these families can be viewed as a qualitative version of their quantitative results (see Remark~\ref{rem:KobayashiSpeh} for details).

The family of operators $\mathbf{A}_{\xi,\lambda}^{\eta,\nu}$ is directly related to a family of symmetry breaking operators between principal series representations studied in \cite{DF24}. It would be interesting to investigate more systematically to which extent the operators in \cite{DF24} give rise to operators between other degenerate series or generalized principal series representations.

We remark that Theorem~\ref{thm:BranchingLawsIntro} only treats half of all Speh representations. The other half can be found as quotients inside the degenerate series $\pi_{\xi,\lambda}$ for $\xi_1\neq\xi_2$ and $\lambda_1=-\lambda_2\in\{-1,-2,\ldots\}$. We exclude this case since the Hermitian form that the Knapp--Stein operators $\mathbf{T}_{\xi,\lambda}$ with $\xi_1\neq\xi_2$ induce is not $\GL(2n,\RR)$- but only $\SL(2n,\RR)$-invariant. Restricted to $\SL(2n,\RR)$, the Speh representations decompose into the direct sum of two irreducible representations, and the Hermitian form has opposite signs on these two components (see \cite{Sah90}). Taking both positive and negative signatures into account, our method can also produce the corresponding branching laws for the remaining Speh representations, but this would require a special treatment of this particular case which is in contrast to the uniform methods presented in this paper. We remark that the alternative proof using the results of Sahi--Stein~\cite{SS90} (see Remark~\ref{rem:Adduced}) also covers this case.

It is also worth mentioning that for rank one groups, branching laws for complementary series representations have been studied before in several case, see \cite{KS15, KS18,MOZ16,MO15,SV11,SV12,SZ16,Wei24}.

\subsection*{Notation}

Throughout the paper we stick to the convention that $\NN=\{0,1,\dots\}$. We use $\RR^{m\times n}$ to denote the space of $m\times n$ matrices and write $I_k\in\RR^{k\times k}$ to denote the identity matrix and $E_{i,j}$ to denote the matrix which contains zero everywhere but at the $(i,j)$-th entry where there is a one. We further use $e_i$ to denote the $i$-th standard basis vector of $\CC^n$. The matrix transpose or other transposes with respect to non-degenerate bilinear forms are denoted by the superscript ${}^\intercal$. For $\varepsilon\in \ZZ/2\ZZ$ and $\mu \in \CC$ we write $|x|^\mu_{\varepsilon}=\sgn(x)^\varepsilon|x|^\mu$ $(x\in \RR^\times)$. 

\subsection*{Acknowledgments} The first author was supported by the Carlsberg Foundation (grant no. CF24-045). The second author was partially supported by the Aarhus University Research Foundation (grant no. AUFF-E-2022-9-34). Both authors acknowledge support of the Institut Henri Poincaré (UAR 839 CNRS-Sorbonne Université), and LabEx CARMIN (ANR-10-LABX-59-01) during the thematic trimester ``Representation Theory and Noncommutative Geometry'' in early 2025.

\section{Principal and degenerate series representations} \label{sec:Representations}

In this section, we define principal series and degenerate series representations of $G=\GL(2n,\RR)$ and $H=\GL(2n-1,\RR)$.

\subsection{Degenerate series representations of \texorpdfstring{$\GL(2n,\RR)$}{GL(2n,R)}}\label{sec:DegSerG}

In this subsection we consider the group $G=\GL(2n,\RR)$ and its maximal parabolic subgroup 
\[
P_G=\begin{pmatrix}
    \GL(n,\RR) & \R^{n\times n}\\
     & \GL(n,\RR)
\end{pmatrix}
\]
which has a Langlands decompositon $P_G=M_GA_GN_G$ with 
\[
M_G=\begin{pmatrix}
    \SL^\pm(n,\RR) & \\
     & \SL^\pm(n,\RR)
\end{pmatrix},\quad A_G=\begin{pmatrix}
    \RR_{+} I_n & \\
    & \RR_{+} I_n
\end{pmatrix},\quad N_G=\begin{pmatrix}
    I_n & \RR^{n\times n} \\
     & I_n
\end{pmatrix}.
\]
Write $\fraka_G$ and $\frakn_G$ for the Lie algebras of $A_G$ and $N_G$.

We consider the irreducible one-dimensional representations of $M_G$ which are given by the characters 
\[
M_G\to \{-1,1\},\quad \diag(g_1,g_2)\mapsto \det(g_1)^{\xi_1}\det(g_2)^{\xi_2}
\]
for $\xi=(\xi_1,\xi_2)\in (\ZZ/2\ZZ)^2$. We identify $\xi$ with the corresponding character of $M_G$. Furthermore, we make the identification $(\fraka_G)_\CC^* \simeq \CC^2$ by the mapping 
$$\lambda\mapsto \big(\lambda(\diag(I_n,0)),\lambda(\diag(0,I_n))\big).$$ Then $\rho_G=\frac{1}{2}\tr \ad |_{\frakn_G}\in(\fraka_G)_\CC^*$ corresponds to $(\frac{n}{2},-\frac{n}{2})\in \CC^2$. For $\lambda \in \CC^2$ we write $e^\lambda$ for the character of $A_G$ given by $e^\lambda(e^H)=e^{\lambda(H)}$ $(H\in \fraka_G)$.

For $\xi\in (\ZZ/2\ZZ)^2$ and $\lambda\in \CC^2$ we extend the character $\xi\otimes e^\lambda$ of $M_GA_G$ trivially to $P_G=M_GA_GN_G$ and write $\xi\otimes e^\lambda\otimes 1$ for it. Using smooth parabolic induction from $P_G$ to $G$ we obtain the degenerate series representation $\pi_{\xi,\lambda}=\Ind_{P_G}^G(\xi\otimes e^\lambda\otimes 1)$ as the left-regular representation of $G$ on 
\[
\{f\in C^\infty(G)\,|\,f(gman)=\xi(m)^{-1}a^{-\lambda-\rho_G}f(g)\,\forall man\in M_GA_GN_G\,\forall g \in G\}.
\]

On $\pi_{\xi,\lambda}\times \pi_{\xi,-\lambda}$ we have a $G$-invariant bilinear pairing given by
\begin{equation}
( f\,|\, f')_G=\int_{G/P_G} f(g)f'(g)\,d(gP_G)\qquad (f\in \pi_{\xi,\lambda}, f'\in \pi_{\xi,-\lambda}),\label{eq:InvPairingG}
\end{equation}
where $d(gP_G)$ denotes the unique (up to scalar multiples) $G$-invariant integral on $C(G\times_{P_G}\CC_{2\rho_G})$, the continuous sections of the density bundle over $G/P_G$ which is the homogeneous vector bundle induced from the character $1\otimes e^{2\rho_G}\otimes 1$ of $P_G$ (see \cite[Theorem 1.1]{CF23}).

\subsection{Degenerate series representations of \texorpdfstring{$\GL(2n-1,\RR)$}{GL(2n-1,R)}}\label{subsec:Hreps}

We now consider the subgroup $H=\GL(2n-1,\RR)$ of $G$ embedded in the upper left corner, and its parabolic subgroup
\[
P_{H}=\begin{pmatrix}
    \GL(n-1,\RR) & \R^{(n-1)\times 1} & \RR^{(n-1)\times(n-1)}\\
     & \GL(1,\RR) & \RR^{1\times (n-1)}\\
     & & \GL(n-1,\RR)
\end{pmatrix}.
\]
This parabolic subgroup has a Langlands decomposition $P_{H}=M_{H}A_{H}N_{H}$ given by
\begin{align*}
    M_{H}&=\begin{pmatrix}
    \SL^\pm(n-1,\RR) & & \\
    & \{\pm1\} & \\
    & & \SL^\pm(n-1,\RR)
    \end{pmatrix},\\
    A_{H}&=\begin{pmatrix}
    \RR_{+} I_{n-1} & & \\
     & \RR_{+} & \\
     & & \RR_{+} I_{n-1}
    \end{pmatrix},\\
    N_{H}&=\begin{pmatrix}
    I_{n-1} & \R^{(n-1)\times 1} & \RR^{(n-1)\times (n-1)}\\
     & 1 & \RR^{1\times (n-1)}\\
     & & I_{n-1}
    \end{pmatrix}.
\end{align*}
Write  $\fraka_H$ and $\frakn_H$ for the Lie algebras of $A_H$ and $N_H$. 

We consider the irreducible one-dimensional representations of $M_H$ that are given by the characters 
\[
M_H\to \{-1,1\},\quad \diag(g_1,\varepsilon,g_2)\mapsto \det(g_1)^{\eta_1}\varepsilon^{\eta_2}\det(g_2)^{\eta_3},
\]
for $\eta=(\eta_1,\eta_2,\eta_3)\in (\ZZ/2\ZZ)^3$. We identify $\eta$ with the corresponding character of $M_H$. Furthermore, we make the identification $(\fraka_H)^*_\CC\simeq \CC^3$ by 
\[
\nu \mapsto \big(\nu(\diag(I_{n-1},0,0)),\nu(\diag(0,1,0)),\nu(\diag(0,0,I_{n-1}))\big).
\]
Then $\rho_H=\frac{1}{2}\tr \ad |_{\frakn_H}$ corresponds to $(\frac{n}{2},0,-\frac{n}{2})\in \CC^3$. For $\nu\in \CC^3$ we write $e^\nu$ for the character of $A_H$ given by $e^\nu(e^H)=e^{\nu(H)}$ $(H\in \fraka_H)$.

For $\eta\in (\ZZ/2\ZZ)^3$ and $\nu\in \CC^3$ we extend the character $\eta\otimes e^\nu$ trivially to $P_H=M_HA_HN_H$ and write $\eta\otimes e^\nu\otimes 1$ for it. Using smooth (normalized) parabolic induction from $P_H$ to $H$ we obtain the degenerate series representation $\tau_{\eta,\nu}=\Ind_{P_H}^H(\eta\otimes e^\nu\otimes 1)$ as the left-regular representation of $H$ on 
\[
\{f\in C^\infty(H)\,|\, f(hman)=\eta(m)^{-1}a^{-\nu-\rho_H}f(h)\,\forall man\in M_HA_HN_H\,\forall h\in H\}.
\]

As in the previous section, we have $\tau_{\eta,\nu}\times \tau_{\eta,-\nu}$ an $H$-invariant bilinear pairing given by
\begin{equation}
(f\,|\,f')_H=\int_{H/P_H}f(h)f'(h)\,d(hP_H)\qquad(f\in \tau_{\eta,\nu}, f'\in \tau_{\eta,-\nu}),\label{eq:InvPairingH}
\end{equation}
where $d(hP_G)$ denotes the unique (up to scalar multiples) $H$-invarant integral on $C(H\times_{P_H}\CC_{2\rho_H})$.

\subsection{Principal series representations of \texorpdfstring{$\GL(2n,\RR)$}{GL(2n,R)} and \texorpdfstring{$\GL(2n-1,\RR)$}{GL(2n-1,R)}}

\noindent Consider the group $G_k=\GL(k,\R)$ and denote its corresponding minimal parabolic consisting of upper triangular matrices by $P_k$. This parabolic has a Langlands decomposition $P_k=M_kA_kN_k$ where $M_k$ is the group of diagonal matrices with entries from $\{-1,1\}$, $A_k$ is the group of diagonal matrices with strictly positive entries and $N_k$ is the group of unipotent  upper triangular matrices. Write $\fraka_k$ and $\frakn_k$ for the Lie algebras of $A_k$ and $N_k$.

The irreducible representations of $M_k$ are one-dimensional and given by the characters 
\[
M_k\to \{-1,1\},\quad \diag(\varepsilon_1,\dots,\varepsilon_k)\mapsto \varepsilon_1^{\delta_1}\cdots \varepsilon_k^{\delta_k},
\]
where $\delta=(\delta_1,\dots,\delta_k)\in (\ZZ/2\ZZ)^k$. We identify $\delta$ with the corresponding character of $M_k$. Furthermore, we make the identification $(\fraka_k)_\CC^*\simeq \CC^k$ by mapping $\mu\mapsto (\mu(E_{1,1}),\dots, \mu(E_{k,k}))$. Then the half sum of the positive roots $\rho_k=\frac{1}{2}\tr \ad |_\frakn$ corresponds to $\frac{1}{2}(k-1,k-3,\dots,3-k,1-k)$. For $\mu\in \CC^k$ we write $e^\mu$ for the character of $A_k$ given by $e^\mu(e^H)=e^{\mu(H)}$ $(H\in \fraka_k)$.

For $\delta\in (\ZZ/2\ZZ)^k$ and $\mu \in \CC^k$ we extend the character $\delta\otimes e^\mu$ of $M_kA_k$ trivially to $P_k=M_kA_kN_k$ and write $\delta\otimes e^\mu\otimes 1$ for it. Using smooth (normalized) parabolic induction from $P_k$ to $G_k$ we obtain the principal series representation $\Ind_{P_k}^{G_k}(\delta\otimes e^\mu\otimes 1)$ as the left-regular representation of $G$ on 
\[
\{f\in C^\infty(G_k)|f(gman)=\delta(m)^{-1} a^{-\mu-\rho} f(g)\, \forall man\in M_kA_kN_k\,\forall g \in G_k\}.
\]

When $k = 2n$ (i.e. $G_k =G= \GL(2n, \RR)$), we denote the principal series representation induced from the minimal parabolic subgroup $P_k$ as
\[
\pi_{\xi',\lambda'}^{\ps} = \Ind_{P_{2n}}^G(\xi' \otimes e^{\lambda'} \otimes 1), \quad (\xi' \in (\ZZ/2\ZZ)^{2n}, \; \lambda' \in \CC^{2n}).
\]
Similarly, for $k = 2n-1$ (i.e. \( G_k=H = \GL(2n-1, \RR) \)), the principal series representation is denoted by
\[
\tau_{\eta',\nu'}^{\ps} = \Ind_{P_{2n-1}}^H(\eta' \otimes e^{\nu'} \otimes 1), \quad (\eta' \in (\ZZ/2\ZZ)^{2n-1}, \; \nu' \in \CC^{2n-1}).
\]
\begin{lemma}\label{lem:degseriesInsidePrincipalSeries}
Let $(\xi',\lambda')\in (\ZZ/2\ZZ)^{2n}\times \CC^{2n}$ and $(\xi,\lambda)\in  (\ZZ/2\ZZ)^2\times \CC^2$. If the parameters are related by
\[
\xi_j'=\xi_1\qquad \xi_{j+n}'=\xi_2,\qquad \lambda_j'=\lambda_1+\frac{n-2j+1}{2},\qquad \lambda_{n+j}'=\lambda_2+\frac{n-2j+1}{2}
\]
for $j=1,\dots,n$, then
\[
\Pi_{\xi',\lambda'}f(g)=\int_{G_n/P_n\times G_n/P_n}|\det(x_1)|_{\xi_1}^{\lambda_1+\frac{n}{2}}|\det(x_2)|^{\lambda_2-\frac{n}{2}}_{\xi_2}f\bigg(g\begin{pmatrix}
x_1 & \\
 & x_2
\end{pmatrix}\bigg)\,dx_1\,dx_2,
\]
where $dx_1$ and $dx_2$ denote the unique $G_n$-invariant integral on $C^\infty(G_n\times_{P_n}\CC_{2\rho_n})$, defines a surjective intertwining operator
$\Pi_{\xi',\lambda'}:\pi_{\xi',\lambda'}^{\ps}\to \pi_{\xi,\lambda}$.
Similarly, if the parameters $(\eta',\nu')\in(\ZZ/2\ZZ)^{2n-1}\times\CC^{2n-1}$ and $(\eta,\nu)\in(\ZZ/2\ZZ)^3\times\CC^3$ are related by
\[
\eta_j'=\eta_1,\qquad \eta_n'=\eta_2,\qquad \eta_{j+n}'=\eta_3
\]
and
\[
\nu_j'=\nu_1+j-\frac{n}{2},\qquad \nu_n'=\nu_2\qquad \nu_{n+j}'=\nu_3+j-\frac{n}{2}
\]
for $j=1,2,\dots,n-1$, we get an injective intertwining map $\iota_{\eta,\nu}:\tau_{\eta,\nu}\to\tau_{\eta',\nu'}^{\ps}$ given by $f\mapsto f$.
\end{lemma}

\begin{proof}
This follows immediately by checking the equivariance properties for the chosen parameters. The integral converges as the integrand is continuous and $G_n/P_n$ is compact.
\end{proof}

\section{Knapp--Stein intertwining operators and unitarizable quotients} \label{sec:KnappSteinAndUnitary}

In this section we define and normalize Knapp--Stein intertwining operators between the degenerate series representations of $G$ and $H$ and use them to identify unitarizable quotients. For this we fix the maximal compact subgroup $K_k=O(k)$ of $G_k$. 

\subsection{Knapp--Stein intertwining operators for \texorpdfstring{$G$}{G}}\label{sec:KSforG}

The Weyl group $N_{K_{2n}}(A_G)/Z_{K_{2n}}(A_G)$ is a group $\{e,w_G\}$ with two elements. Using the identification of $(\ZZ/2\ZZ)^2$ with the characters of $M_G$ and $(\fraka_G)^*_\CC\simeq \CC^2$ it follows that $w_G(\xi_1,\xi_2)=(\xi_2,\xi_1)$ and $w_G(\lambda_1,\lambda_2)=(\lambda_2,\lambda_1)$. We can pick the following representative of $w_G=[\Tilde{w}_G]$:
\[
\Tilde{w}_G=\begin{pmatrix}
    0 & -I_n\\
    I_n & 0
\end{pmatrix}
\]
and let $\overline{N}_G=N_G^\intercal=\tilde{w}_G N_G\tilde{w}_G^{-1}$.

For $f\in \pi_{\xi,\lambda}$ the integral
\[
T_{\xi,\lambda}f(g)=\int_{\overline{N}_G}f(g\Tilde{w}_G\overline{n})\,d\overline{n}\qquad (g\in G),
\]
converges absolutely in some range of parameters $\lambda\in \CC^2$ and defines an intertwining operator 
\[
T_{\xi,\lambda}:\pi_{\xi,\lambda}\to \pi_{w_G(\xi,\lambda)}
\]
known as the Knapp--Stein intertwining operator (see e.g. \cite{K86}). It can be shown that the family $T_{\xi,\lambda}$ has a meromorphic extension to all $\lambda\in \CC^2$ (if viewed as an operator in the compact picture where the representation space is independent of $\lambda$). Note that choosing a different representative $\Tilde{w}_G$ would only change the operator by a sign. 

On the open dense subset of $\overline{N}_G$ consisting of 
\[
\overline{n} = 
\begin{pmatrix}
    I_n & \\
    X & I_n
\end{pmatrix}
\qquad \text{with $X \in \mathbb{R}^{n \times n}$ invertible,}
\]
we can decompose
\[
w_G \overline{n} =
\begin{pmatrix}
    -X & -I_n \\
    I_n & 
\end{pmatrix} =
\begin{pmatrix}
    I_n & \\
    -X^{-1} & I_n
\end{pmatrix}
\begin{pmatrix}
    -X & \\
    & -X^{-1}
\end{pmatrix}
\begin{pmatrix}
    I_n & X^{-1} \\
    & I_n
\end{pmatrix},
\]
allowing us to write the Knapp--Stein operator as 
\begin{align*}
T_{\xi ,\lambda } f(g)&=\int_{\RR^{n\times n}}|\det(X)|_{\xi _1+\xi _2}^{\lambda_1 -\lambda_2 -n}f(g\begin{pmatrix}
    I_n & \\
    X & I_n
\end{pmatrix})\,dX,
\end{align*}
where we have used the equivariance of functions in $f\in \pi_{\xi,\lambda}$ and made the change of variables $X\to -X^{-1}$. By \cite[Lemma 3.16]{BSZ06} this integral converges absolutely for $\Re(\lambda_1-\lambda_2)>n-1$. Consider the following normalization
\begin{equation}
\mathbf{T}_{\xi,\lambda} =\frac{T_{\xi,\lambda}}{n_\mathbf{T}(\xi,\lambda)},\qquad \text{where}\qquad n_\mathbf{T}(\xi,\lambda)=\prod_{i=1}^n\Gamma\big(\tfrac{\lambda_1-\lambda_2-n+i+[\xi_1+\xi_2]}{2}\big).\label{eq:HolomorphicT}
\end{equation}

\begin{proposition}[{see \cite[Theorem 5.12]{BSZ06}}] \label{prop:storTonyholomorphic}
For $\xi \in (\ZZ/2\ZZ)^2$ with $\xi_1=\xi_2$ the family $\mathbf{T}_{\xi,\lambda}$ depends holomorphically on $\lambda\in\C^2$.
\end{proposition}

Note that by \cite[Proposition 14.10 and Remark thereafter]{K86} together with \eqref{eq:HolomorphicT}, the adjoint of $\mathbf{T}_{\xi,\lambda}$ with respect to the invariant pairing $(\cdot\,|\,\cdot)_G$ from \eqref{eq:InvPairingG} is given by $\mathbf{T}_{w_G(\xi,-\lambda)}$, i.e.
\begin{equation}\label{eq:TonySymmetricG}
    ( \mathbf{T}_{\xi,\lambda}f\,|\, f')_G = ( f\,|\, \mathbf{T}_{w_G(\xi,-\lambda)}f')_G \qquad (f\in\pi_{\xi,\lambda},f'\in\pi_{w_G(\xi,-\lambda)}).
\end{equation}

\subsection{Unitary representations of \texorpdfstring{$G$}{G}}\label{sec:UnirrepsG}

For $\xi\in(\ZZ/2\ZZ)^2$ and $\lambda\in\RR^2$ with $\xi_1=\xi_2$ and $\lambda_1+\lambda_2=0$ we can use the invariant pairing $(\cdot\,|\,\cdot)_G$ of \eqref{eq:InvPairingG} to define an invariant Hermitian form on $\pi_{\xi,\lambda}$ by
\[
\langle f_1,f_2\rangle_{\xi,\lambda}=(f_1\,|\,\overline{\mathbf{T}_{\xi,\lambda}f_2})_G.
\]

\begin{theorem}[{see \cite[Proposition 7.7 and 7.8]{BSZ06}, \cite[Theorems 4.A, 4.B, 4.D and 4.E]{Sah95}, \cite{SS90}}]\label{thm:BigHermitianFormPositiveSemidefinite}
For $\xi\in(\ZZ/2\ZZ)^2$ and $\lambda\in\RR^2$ with $\xi_1=\xi_2$ and $\lambda_1+\lambda_2=0$, the invariant Hermitian form $\langle\,\cdot \,,\,\cdot \,\rangle_{\xi,\lambda}$ on $\pi_{\xi,\lambda}$ is positive semidefinite if and only if
$$ \lambda_1=-\lambda_2 \in \{\ldots,-\tfrac{5}{2},-\tfrac{3}{2},-\tfrac{1}{2}\}\cup(-\tfrac{1}{2},\tfrac{1}{2})\cup\{\tfrac{1}{2},1,\ldots,\tfrac{n}{2}\}. $$
\end{theorem}

This allows us to identify four different kinds of unitary representations for $\xi_1=\xi_2$:
\begin{enumerate}
    \item \emph{Unitary degenerate series} $\pi_{\xi,\lambda}$ for $\lambda\in i\RR^2$.
    \item \emph{Stein's complementary series} $\pi_{\xi,\lambda}^{\operatorname{c.s.}}=\pi_{\xi,\lambda}$ for $\lambda_1=-\lambda_2\in(-\frac{1}{2},0)$.
    \item \emph{Small irreducible quotients} $\pi_{\xi,\lambda}^{\operatorname{small}}=\pi_{\xi,\lambda}/\ker(\mathbf{T}_{\xi,\lambda})$ for $\lambda_1=-\lambda_2\in \{\tfrac{1}{2},1,\dots,\frac{n}{2}\}$.
    \item \emph{Speh representations} $\pi_{\xi,\lambda}^{\operatorname{Speh}}=\pi_{\xi,\lambda}/\ker(\mathbf{T}_{\xi,\lambda})$ for $\lambda_1=-\lambda_2\in \{-\frac{1}{2},-\frac{3}{2},-\frac{5}{2},\ldots\}$.
\end{enumerate}

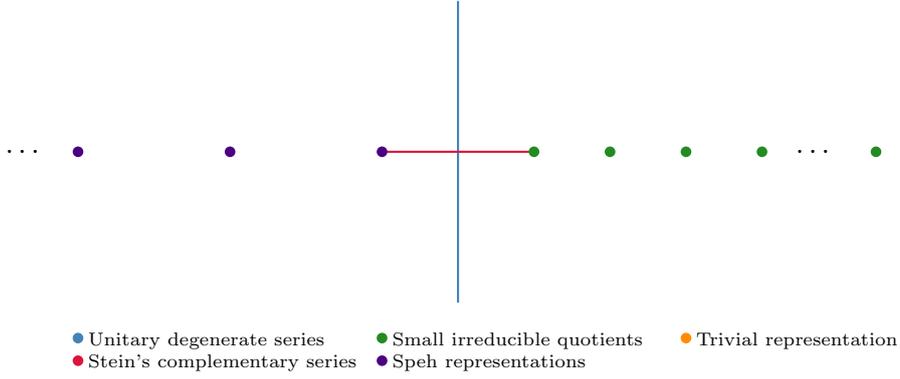
\begin{figure}\label{fig:UnirrepsG}
    \centering
    \begin{tikzpicture}

    \definecolor{imaginaryblue}{RGB}{70,130,180} 
    \definecolor{realred}{RGB}{220,20,60} 
    \definecolor{dotgreen}{RGB}{34,139,34} 
    \definecolor{dotorange}{RGB}{255,140,0} 
    \definecolor{dotpurple}{RGB}{75,0,130}

    \draw[thick,imaginaryblue] (0,-2) -- (0,2);

    \draw[realred,thick] (0,0) -- (1,0);
    \draw[realred,thick] (0,0) -- (-1,0);
 \foreach \x in {1,3,5} {
        \fill[dotpurple] (-\x,0) circle(2pt);
    }

    \foreach \x in {1,2,3,4} {
        \fill[dotgreen] (\x,0) circle(2pt);
    }
    \fill[dotgreen] (5.5,0) circle(2pt);
    \node at (4.7,0) {\dots};
    \node at (-5.7,0) {\dots};
    \fill[dotorange] (6.5,0) circle(2pt);

    \fill[imaginaryblue] (-5,-2.47) circle (2pt);
    \node[right] at (-5,-2.5) {\tiny Unitary degenerate series};
     \fill[realred] (-5,-2.77) circle(2pt);
     \node[right] at (-5,-2.8) {\tiny Stein's complementary series};
    \fill[dotgreen] (-1,-2.47) circle(2pt);
     \node[right] at (-1,-2.5) {\tiny Small irreducible quotients};
     \fill[dotpurple] (-1,-2.77) circle(2pt);
     \node[right] at (-1,-2.8) {\tiny Speh representations};
     \fill[dotorange] (3,-2.47) circle(2pt);
     \node[right] at (3,-2.5) {\tiny Trivial representation};
\end{tikzpicture}
    \caption{The unitary representations related to $\pi_{\xi,\lambda}$ plotted for the parameters $\lambda_1-\lambda_2\in \CC$.}
\end{figure}

Note that the last small irreducible quotient $\pi_{\xi,\lambda}^{\operatorname{small}}$ for $\xi_1=\xi_2$ and $\lambda_1=-\lambda_2=\frac{n}{2}$ is the trivial representation.

\begin{remark}
    Using the results of \cite{SS90} it can be shown that all Speh representations have the same Gelfand--Kirillov dimension as the full degenerate series $\pi_{\xi,\lambda}$, namely $n^2$. In that sense, they can be considered \emph{large} irreducible quotients of $\pi_{\xi,\lambda}$. In contrast to this, the results of \cite{BSS90} and \cite{BSZ06} imply that the Gelfand--Kirillov dimension of $\pi_{\xi,\lambda}^{\mathrm{small}}$ for $\lambda_1=-\lambda_2=\frac{n-k}{2}$, $k\in\{0,\ldots,n-1\}$, is equal to the dimension of the subvariety of $\RR^{n\times n}$ consisting of matrices of rank at most $k$, thus justifying the name \emph{small} irreducible quotients.
\end{remark}

\subsection{Knapp--Stein intertwining operators for \texorpdfstring{$H$}{H}}\label{sec:KSforH}

Completely analogously to the case for $\pi_{\xi,\lambda}$ we define the Knapp--Stein operators for $\tau_{\eta,\nu}$. We denote by $w_H$ the element of the Weyl group $N_{K_{2n-1}}(A_H)/Z_{K_{2n-1}}(A_H)$ that acts on characters $\eta\in(\ZZ/2\ZZ)^3$ of $M_H$ and $\nu\in(\fraka_H)^*_\CC\simeq \CC^3$ by $w_H(\eta_1,\eta_2,\eta_3)=(\eta_3,\eta_2,\eta_1)$ and $w_H(\nu_1,\nu_2,\nu_3)=(\nu_3,\nu_2,\nu_1)$. It has a representative $\Tilde{w}_H\in K_{2n-1}$ given by
\[
\Tilde{w}_H=\begin{pmatrix}
    0 & 0 & -I_{n-1}\\
    0 & 1 & 0\\
    I_{n-1} & 0 & 0
\end{pmatrix}.
\]
We further let $\overline{N}_H=N_H^\intercal=\tilde{w}_HN_H\tilde{w}_H^{-1}$.

For $f\in \tau_{\eta,\nu}$ the integral
\[
S_{\eta,\nu}f(h)=\int_{\overline{N}_H}f(hw_H\overline{n})d\overline{n}\qquad (h\in H),
\]
converges absolutely in some range of parameters $\nu \in \CC^3$ and defines the Knapp--Stein intertwining operator 
\[
S_{\eta,\nu}:\tau_{\eta,\nu}\to \tau_{w_{H}(\eta,\nu)}.
\]
On the open dense subset of $\overline{N}_H$ consisting of 
\[
\overline{n}=\begin{pmatrix}
    I_{n-1} & & \\
    v^\intercal & 1 & \\
    X & u & I_{n-1}
\end{pmatrix}\qquad \text{with $X$ invertible and  $v^\intercal X^{-1}u\neq 1$,}
\]
we can decompose  
\begin{equation*}
w_H\overline{n}
=\begin{pmatrix}
    I_{n-1} & & \\
    -v^\intercal\gamma_1^{-1} & 1 & \\
    -\gamma_1^{-1} & -\gamma_1^{-1}\gamma_2^{-1}u & I_{n-1}
\end{pmatrix}
\begin{pmatrix}
    -\gamma_1 & & \\
    & \gamma_2 & \\
    & & -\gamma_3
\end{pmatrix}
\begin{pmatrix}
    I_{n-1} & \gamma_1^{-1}u & \gamma_1^{-1}\\
     & 1 & -\gamma_2^{-1}v^\intercal X^{-1}\\
     & & I_{n-1}
\end{pmatrix},
\end{equation*}
where 
\[
\gamma_1=X,\qquad \gamma_2=1-v^\intercal X^{-1}u,\qquad \gamma_3=X^{-1}\Big(I_{n-1}+\frac{uv^\intercal X^{-1}}{1-v^\intercal X^{-1}u}\Big).
\]
Using this and making the substitutions $X\mapsto -X^{-1}$, $v^\intercal\to -v^\intercal\gamma_1^{-1}$ and $u\to -\gamma_1^{-1}\gamma_2^{-1}u$ we get 
\begin{multline*}
S _{\eta ,\nu }f(h)=\int_{\RR^{(n-1)\times(n-1)}\times \RR^{n-1}\times \RR^{n-1}}|\det(X)|^{\nu _1-\nu _3-n}_{\eta _1+\eta _3}\\ \times|1-v^\intercal X^{-1}u|^{\nu _2-\nu _3-\frac{n}{2}}_{\eta _2+\eta _3}f\bigg(h\begin{pmatrix}
    1 & & \\
    v^\intercal & 1 &\\
    X & u & 1
\end{pmatrix}\bigg) d(X,u,v).
\end{multline*}
This integral converges absolutely whenever $\Re(\nu_1-\nu_2),\Re(\nu_2-\nu_3)>\frac{n}{2}-1$. (This can be seen by embedding $\tau_{\eta,\nu}$ inside $\tau_{\eta',\nu'}^{\ps}$ by Lemma~\ref{lem:degseriesInsidePrincipalSeries} and decomposing the Knapp--Stein intertwining operator with respect to $w_H$ in terms of Knapp--Stein intertwining operators with respect to simple reflections for which the convergence is well known, see e.g. \cite[Theorem 7.12]{K86}.)
Consider the following normalization
\begin{equation}
n_{\mathbf{S}}(\eta,\nu)=\Gamma(\tfrac{\nu_1-\nu_2+1-\frac{n}{2}+[\eta_1+\eta_2]}{2})\Gamma(\tfrac{\nu_2-\nu_3+1-\frac{n}{2}+[\eta_2+\eta_3]}{2})\prod_{i=2}^n\Gamma(\tfrac{\nu_1-\nu_3+i-n+[\eta_1+\eta_3]}{2})\label{eq:HolomorphicS}
\end{equation}
and set $\mathbf{S}_{\eta,\nu}=n_{\mathbf{S}}(\eta,\nu)^{-1}S_{\eta,\nu}$.

\begin{proposition}\label{prop:lilleTonyholomorphic}
The family $\mathbf{S}_{\eta,\nu}$ depends holomorphically on $\nu\in \C^3$ for all $\eta\in (\ZZ/2\ZZ)^3$.
\end{proposition}

\begin{proof}
    This follows from the factorization for $\mathbf{S}_{\eta,\nu}$ obtained in Lemma~\ref{lemma:FactorizationOfS} in the next section.
\end{proof}

Similar to \eqref{eq:TonySymmetricG}, a careful inspection of the normalization \eqref{eq:HolomorphicS} shows that the Knapp--Stein intertwiner $\mathbf{S}_{\eta,\nu}$ satisfies
\begin{equation}\label{eq:TonySymmetricH}
    ( \mathbf{S}_{\eta,\nu}f\,|\, f')_H = ( f\,|\, \mathbf{S}_{w_H(\eta,-\nu)}f')_H \qquad (f\in\tau_{\eta,\nu},f'\in\tau_{w_H(\eta,-\nu)}).
\end{equation}

\subsection{Unitary representations of \texorpdfstring{$H$}{H}}\label{sec:UnirrepsH}

For $\eta\in(\ZZ/2\ZZ)^3$ and $\nu\in\CC^3$ with $\eta_1=\eta_3$ and $\nu_1=-\nu_3\in\RR$, $\nu_2\in i\RR$ we can use the invariant pairing $(\cdot\,|\,\cdot)_H$ of \eqref{eq:InvPairingH} to define an invariant Hermitian form on $\tau_{\eta,\nu}$ by
\[
\langle f_1,f_2\rangle_{\eta,\nu}=(f_1\,|\,\overline{\mathbf{S}_{\eta,\nu}f_2})_H.
\]

To identify complementary series and irreducible unitarizable quotients of $\tau_{\eta,\nu}$, we first use induction in stages to relate the representations $\tau_{\eta,\nu}$ to degenerate series representations induced from the maximal parabolic subgroup
$$ Q_H = \begin{pmatrix}\GL(2n-2,\RR)&\RR^{(2n-2)\times1}\\&\GL(1,\RR)\end{pmatrix}. $$
We use notation anologous to the one in Section~\ref{sec:UnirrepsG} and denote by $\varpi_{\xi,\lambda}$, $\varpi_{\xi,\lambda}^{\textrm{c.s.}}$, $\varpi_{\xi,\lambda}^{\textrm{small}}$ and $\varpi_{\xi,\lambda}^{\textrm{Speh}}$ the degenerate series, Stein's complementary series, small representations and Speh representations of $\GL(2n-2,\RR)$, respectively. Moreover, let
$$ \mathbf{R}_{\xi,\lambda}:\varpi_{(\xi_1,\xi_2),(\lambda_1,\lambda_2)}\to\varpi_{(\xi_2,\xi_1),(\lambda_2,\lambda_1)} $$
be the family of Knapp--Stein intertwining operators analogous to $\mathbf{T}_{\xi,\lambda}$ but for $\GL(2n-2,\RR)$ instead of $\GL(2n,\RR)$.

Using slightly different notation as before, we write $\Ind_{Q_H}^H(\varpi\otimes\chi_{\varepsilon,s})$ for the representation parabolically induced from the representation of $Q_H$ given by $\varpi$ on $\GL(2n-2,\RR)$ and by the character $\chi_{\varepsilon,s}(x)=|x|_\varepsilon^s$ on $\GL(1,\RR)$ (normalized induction). A similar notation is also used for induction from $P_H$ in the next proof.

\begin{lemma}\label{lemma:FactorizationOfS}
    There is a holomorphic family of standard intertwining operators
    $$ \mathbf{U}_{\eta,\nu}:\tau_{\eta,\nu} \to \Ind_{Q_H}^H(\varpi_{(\eta_1,\eta_3),(\nu_1,\nu_3)}\otimes\chi_{\eta_2,\nu_2}) $$
    such that whenever $\nu_2-\nu_3\not\in\ZZ$, the operator $\mathbf{U}_{\eta,\nu}$ is an isomorphism and
    \begin{equation}
        \mathbf{S}_{\eta,\nu}=\mathbf{U}_{w_H(\eta,-\nu)}^\intercal\circ\Ind_{Q_H}^H(\mathbf{R}_{(\eta_1,\eta_3),(\nu_1,\nu_3)}\otimes\id_{\chi_{\eta_2,\nu_2}})\circ\mathbf{U}_{\eta,\nu},\label{eq:FactorizationNormalized}
    \end{equation}
    where $\mathbf{U}_{w_H(\eta,-\nu)}^\intercal:\Ind_{Q_H}^H(\varpi_{(\eta_3,\eta_1),(\nu_3,\nu_1)}\otimes\chi_{\eta_2,\nu_2})\to\tau_{w_H(\eta,\nu)}$
    denotes the intertwining operator dual to $\mathbf{U}_{w_H(\eta,-\nu)}$ with respect to the $H$-invariant pairings \eqref{eq:InvPairingH}. In particular, for $\nu_1-\nu_2,\nu_2-\nu_3\not\in\ZZ$,
    $$ \tau_{\eta,\nu}/\ker(\mathbf{S}_{\eta,\nu}) \simeq \Ind_{Q_H}^H\Big(\big[\varpi_{(\eta_1,\eta_3),(\nu_1,\nu_3)}/\ker(\mathbf{R}_{(\eta_1,\eta_3),(\nu_1,\nu_3)})\big]\otimes\chi_{\eta_2,\nu_2}\Big). $$
\end{lemma}

The notation $\Ind_{Q_H}^H(\mathbf{R}_{(\eta_1,\eta_3),(\nu_1,\nu_3)}\otimes\id_{\chi_{\eta_2,\nu_2}})$ is alluding to the fact that parabolic induction is a functor.

\begin{proof}
    Let
    $$ U_{\eta,\nu}f(h) = \int_{\overline{N}_H\cap\tilde{w}^{-1}N_H\tilde{w}} f(h\tilde{w}\overline{n})\,d\overline{n}, \qquad \mbox{where }\tilde{w} = \begin{pmatrix}I_{n-1}&&\\&&I_{n-1}\\&1&\end{pmatrix}. $$
    Then $U_{\eta,\nu}:\tau_{\eta,\nu}=\Ind_{P_H}^H(\chi_{\eta_1,\nu_1}\otimes\chi_{\eta_2,\nu_2}\otimes\chi_{\eta_3,\nu_3})\to\Ind_{P_H'}^H(\chi_{\eta_1,\nu_1}\otimes\chi_{\eta_3,\nu_3}\otimes\chi_{\eta_2,\nu_2})$ is a family of intertwining operators, where $P_H'$ is the standard parabolic subgroup with Levi factor $\GL(n-1,\RR)\times\GL(n-1,\RR)\times\GL(1,\RR)$ (with the blocks sitting on the diagonal in this order). Unlike $P_H$, we have that $P_H'\subseteq Q_H$ and thus, by induction in stages (see e.g. \cite[Chapter VII, \S2]{K86}),
    $$ \Ind_{P_H'}^H(\chi_{\eta_1,\nu_1}\otimes\chi_{\eta_3,\nu_3}\otimes\chi_{\eta_2,\nu_2}) \simeq \Ind_{Q_H}^H(\varpi_{(\eta_1,\eta_3),(\nu_1,\nu_3)}\otimes\chi_{\eta_2,\nu_2}). $$
    Note that
    $$ \tilde{w}_H = \tilde{w}^\intercal\begin{pmatrix}&-I_{n-1}&\\I_{n-1}&&\\&&1\end{pmatrix}\tilde{w} $$
    is a minimal decomposition of the Weyl group element $w_H=[\tilde{w}_H]$, so we can apply the factorization identity for intertwining operators (see e.g. \cite[Proposition 7.10 and Theorem 8.38~(f)]{K86}). The first and the last element on the right hand side give rise to the intertwining operators $U_{w_H(\eta,-\nu)}^\intercal$ and $U_{\eta,\nu}$. By the same arguments as in \cite[Chapter I, Section D]{Spe77}, the intertwining operator for the second element on the right hand side becomes the intertwining operator $\Ind_{Q_H}^H(R_{(\eta_1,\eta_3),(\nu_1,\nu_3)}\otimes\id_{\chi_{\eta_2,\nu_2}})$, where $R_{(\eta_1,\eta_3),(\nu_1,\nu_3)}:\varpi_{(\eta_1,\eta_3),(\nu_1,\nu_3)}\to\varpi_{(\eta_3,\eta_1),(\nu_3,\nu_1)}$ is the (unnormalized) intertwiner for $\GL(2n-2,\RR)$ associated with the Weyl group element
    $$ \begin{pmatrix}&-I_{n-1}\\I_{n-1}&\end{pmatrix}. $$
    Note that $R_{(\eta_1,\eta_3),(\nu_1,\nu_3)}$ is defined analogous to $T_{\xi,\lambda}$. Now the factorization identity for intertwining operators reads
    \begin{equation}
        S_{\eta,\nu} = U_{w_H(\eta,-\nu)}^\intercal\circ\Ind_{Q_H}^H(R_{(\eta_1,\eta_3),(\nu_1,\nu_3)}\otimes\id_{\chi_{\eta_2,\nu_2}})\circ U_{\eta,\nu}.\label{eq:FactorizationUnnormalized}
    \end{equation}
    By a computation similar to the one in Section~\ref{sec:KSforH} for $S_{\eta,\nu}$, we find that
    $$ U_{\eta,\nu}f(h) = (-1)^{(n-1)\eta_3}\int_{\RR^{n-1}}|x_{n-1}|_{\eta_2+\eta_3}^{\nu_2-\nu_3-\frac{n}{2}}f\bigg(h\begin{pmatrix}I_{n-1}&&\\&1&\\&x&I_{n-1}\end{pmatrix}\bigg)\,dx, $$
    so by classical results for the Riesz distributions (see e.g. \cite{GS64}) the renormalization $\mathbf{U}_{\eta,\nu}=\Gamma(\tfrac{\nu_2-\nu_3+1-\frac{n}{2}+[\eta_2+\eta_3]}{2})^{-1}U_{\eta,\nu}$ is holomorphic and $\mathbf{U}_{\eta,\nu}$ is an isomorphism whenever
    $$ \nu_2-\nu_3\not\in\begin{cases}2\ZZ+1&\mbox{if }\eta_2=\eta_3,\\2\ZZ\setminus\{0\}&\mbox{if }\eta_2\neq\eta_3.\end{cases} $$
    (A more detailed account of these maximally degenerate series representations and their intertwining operators can be found in \cite{vDM99}.) It remains to show that the factorization identity \eqref{eq:FactorizationUnnormalized} becomes \eqref{eq:FactorizationNormalized} when replacing every operator by its normalized version. But $n_{\mathbf{S}}(\eta,\nu)$ is the product of the normalization $\Gamma(\tfrac{\nu_2-\nu_3+1-\frac{n}{2}+[\eta_2+\eta_3]}{2})$ of $\mathbf{U}_{\eta,\nu}$, the normalization $\Gamma(\tfrac{\nu_1-\nu_2+1-\frac{n}{2}+[\eta_1+\eta_2]}{2})$ of $\mathbf{U}_{w_H(\eta,-\nu)}^\intercal$ and the normalization $\prod_{i=2}^n\Gamma(\tfrac{\nu_1-\nu_3+i-n+[\eta_1+\eta_3]}{2})$ of $\mathbf{R}_{(\eta_1,\eta_3),(\nu_1,\nu_3)}$, so the proof is complete.
\end{proof}

Now, Theorem~\ref{thm:BigHermitianFormPositiveSemidefinite} allows us to determine the parameters for which the Hermitian form $\langle\cdot,\cdot\rangle_{\eta,\nu}$ on $\tau_{\eta,\nu}$ is positive (semi-)definite.

\begin{theorem}\label{thm:SmallInvariantHermitianFormPositiveSemidefinite}
    For $\eta\in(\ZZ/2\ZZ)^3$ and $\nu\in\CC^3$ with $\eta_1=\eta_3$ and $\nu_1=-\nu_3\in\RR$, $\nu_2\in i\RR$, the invariant Hermitian form $\langle\cdot,\cdot\rangle_{\eta,\nu}$ on $\tau_{\eta,\nu}$ is positive semidefinite for
    $$ \nu_1=-\nu_3 \in \{\ldots,-\tfrac{5}{2},-\tfrac{3}{2},-\tfrac{1}{2}\}\cup(-\tfrac{1}{2},\tfrac{1}{2})\cup\{\tfrac{1}{2},1,\ldots,\tfrac{n}{2}\} $$
    and under the additional assumption that $\nu_2\neq0$ the corresponding unitarizable quotient $\tau_{\eta,\nu}/\ker(\mathbf{S}_{\eta,\nu})$ is irreducible and isomorphic to
    $$ \Ind_{Q_H}^H\Big(\big[\varpi_{(\eta_1,\eta_3),(\nu_1,\nu_3)}/\ker(\mathbf{R}_{(\eta_1,\eta_3),(\nu_1,\nu_3)})\big]\otimes\chi_{\eta_2,\nu_2}\Big). $$
\end{theorem}

\begin{proof}
    Assume first that $\nu_2\in i\RR\setminus\{0\}$. By the previous lemma, in particular \eqref{eq:FactorizationNormalized}, we can conclude that $\langle f_1,f_2\rangle_{\eta,\nu}$ is given by
    $$ \int_{H/Q_H} \Big(\mathbf{U}_{\eta,\nu}f_1(hQ_H)\,\Big|\,\overline{\big(\mathbf{R}_{(\eta_1,\eta_3),(\nu_1,\nu_3)}\otimes\id_{\chi_{\eta_2,\nu_2}}\big)\mathbf{U}_{\eta,\nu}f_2(hQ_H)}\Big)_{\GL(2n-2,\RR)}\,d(hQ_H), $$
    where $(\cdot\,|\,\cdot)_{\GL(2n-2,\RR)}$ denotes the invariant form on $\varpi_{(\eta_1,\eta_3),(\nu_1,\nu_3)}\times\varpi_{(\eta_1,\eta_3),(-\nu_1,-\nu_3)}$. Now the result follows immediately from Theorem~\ref{thm:BigHermitianFormPositiveSemidefinite}. The case $\nu_2=0$ is a simple limit argument by $\nu_2\to0$.

    The second part follows from the old result of Gelfand and Naimark that unitary parabolic induction for general linear groups preserves irreducibility (see e.g. \cite[Proposition 1.1]{Spe81}) and the fact that the unitarizable quotients in Theorem~\ref{thm:BigHermitianFormPositiveSemidefinite} are irreducible (applied to $\GL(2n-2,\RR)$ instead of $\GL(2n,\RR)$).
\end{proof}

\begin{remark}
It might be possible to extend the previous result to the case $\nu_2=0$. However, since the representations all show up as part of a direct integral over $\nu_2\in i\RR$ and $\{0\}$ has measure zero in $i\RR$, we can simply neglect the case $\nu_2=0$ to obtain a decomposition into irreducible representations.
\end{remark}

\begin{figure}\label{fig:UnirrepsH}
    \centering
    \begin{tikzpicture}

    \definecolor{imaginaryblue}{RGB}{70,130,180}
    \definecolor{realred}{RGB}{220,20,60}
    \definecolor{dotgreen}{RGB}{34,139,34}
    \definecolor{dotorange}{RGB}{255,140,0}
    \definecolor{dotpurple}{RGB}{75,0,130}

    \fill[realred,fill opacity=0.5] (-1, -2) rectangle (1, 2);
    \draw[thick,imaginaryblue] (0,-2) -- (0,2);
    \foreach \x in {1,3,5} {
        \draw[thick,dotpurple] (-\x,-2) -- (-\x,2);
    }
    \foreach \x in {1,2,3,4} {
        \draw[thick,dotgreen] (\x,-2) -- (\x,2);
    }
    \draw[thick,dotgreen] (5.5,-2) -- (5.5,2);
    \node at (4.7,0) {\dots};
    \node at (-5.7,0) {\dots};

    \draw[thick,dotorange] (6.5,-2) -- (6.5, 2);
    
    \fill[imaginaryblue] (-5,-2.47) circle (2pt);
    \node[right] at (-5,-2.5) {\tiny Unitary degenerate series};
     \fill[realred] (-5,-2.77) circle(2pt);
     \node[right] at (-5,-2.8) {\tiny Complementary series};
    \fill[dotgreen] (-1,-2.47) circle(2pt);
     \node[right] at (-1,-2.5) {\tiny Small irreducible quotients};
     \fill[dotpurple] (-1,-2.77) circle(2pt);
     \node[right] at (-1,-2.8) {\tiny Large irreducible quotients};
     \fill[dotorange] (3,-2.47) circle(2pt);
     \node[right] at (3,-2.5) {\tiny Unitary characters};
\end{tikzpicture}
    \caption{The unitary representations related to $\tau_{\eta,\nu}$ plotted with $\nu_1-\nu_3 \in \RR$ on the first axis and $\nu_2\in i\RR$ on the second axis.} 
\end{figure}
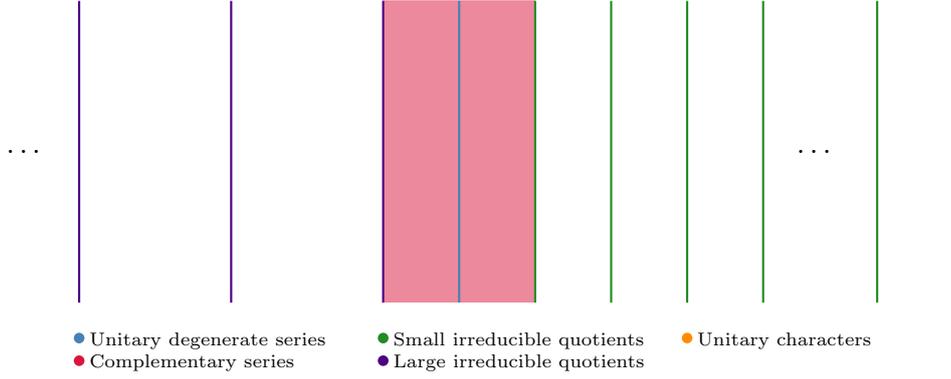

By the previous theorem, we obtain the following list of irreducible unitary representations of $H$ for $\eta_1=\eta_3$ together with isomorphisms to parabolically induced representations from $Q_H$:
    \begin{enumerate}
        \item \emph{Unitary degenerate series} $\tau_{\eta,\nu}$ for $\nu\in i\RR^3$:
        $$ \tau_{\eta,\nu} \simeq \Ind_{Q_H}^H(\varpi_{(\eta_1,\eta_3),(\nu_1,\nu_3)}\otimes\chi_{\eta_2,\nu_2}). $$
        \item \emph{Complementary series} $\tau_{\eta,\nu}^{\mathrm{c.s.}}=\tau_{\eta,\nu}$ for  $\nu_2\in i\RR$ and $\nu_1=-\nu_3\in(-\frac{1}{2},0)$:
        $$ \tau_{\eta,\nu}^{\textrm{c.s.}} \simeq \Ind_{Q_H}^H(\varpi_{(\eta_1,\eta_3),(\nu_1,\nu_3)}^{\textrm{c.s.}}\otimes\chi_{\eta_2,\nu_2}). $$
        \item \emph{Small irreducible quotients} $\tau_{\eta,\nu}^{\operatorname{small}}=\tau_{\eta,\nu}/\ker(\mathbf{S}_{\eta,\nu})$ for $\nu_2\in i\RR$ and $\nu_1=-\nu_3\in \{\frac{1}{2},1,\ldots,\frac{n-1}{2}\}$:
        $$ \tau_{\eta,\nu}^{\textrm{small}} \simeq \Ind_{Q_H}^H(\varpi_{(\eta_1,\eta_3),(\nu_1,\nu_3)}^{\textrm{small}}\otimes\chi_{\eta_2,\nu_2}). $$
        \item \emph{Large irreducible quotients} $\tau_{\eta,\nu}^{\operatorname{large}}=\tau_{\eta,\nu}/\ker(\mathbf{S}_{\eta,\nu})$ for $\nu_2\in i\RR$ and $\nu_1=-\nu_3\in \{-\frac{1}{2},-\frac{3}{2},-\frac{5}{2},\ldots\}$:
        $$ \tau_{\eta,\nu}^{\textrm{large}} \simeq \Ind_{Q_H}^H(\varpi_{(\eta_1,\eta_3),(\nu_1,\nu_3)}^{\textrm{Speh}}\otimes\chi_{\eta_2,\nu_2}). $$
    \end{enumerate}

For $\nu_2\in i\RR$ and $\nu_1=-\nu_3=\frac{n}{2}$ the quotients are unitary characters:
$$ \tau_{\eta,\nu}/\ker(\mathbf{S}_{\eta,\nu})\simeq|\det|_{\eta_2}^{\nu_2}. $$

\section{Symmetry breaking operators}

We construct and analyze two different families of symmetry breaking operators between the degenerate series representations $\pi_{\xi,\lambda}$ of $G$ and $\tau_{\eta,\nu}$ of $H$. These families are related to each other by two functional equations involving the Knapp--Stein operators $\mathbf{T}_{\xi,\lambda}$ and $\mathbf{S}_{\eta,\nu}$. To prove the functional equations, we first show an integral formula for $G/P_G$ in terms of $H/P_H$.

\subsection{First type of symmetry breaking operators} \label{subsec:A-SBO}

In \cite{DF24} the authors studied a meromorphic family of intertwining operators
$$ A_{\xi',\lambda'}^{\eta',\nu'}\in \Hom_H(\pi_{\xi',\lambda'}^{\ps}|_H,\tau_{\eta',\nu'}^{\ps}),$$
which can be described in terms of their distributional kernels $K_{\xi',\lambda'}^{\eta',\nu'}\in \calD'(G)$ by 
\[
 A_{\xi',\lambda'}^{\eta',\nu'}f(h)=\int_{G/P_{2n}}K_{\xi',\lambda'}^{\eta',\nu'}(h^{-1}g)f(g)\,d(gP_{2n}).
\]
The kernels are given by
$$ K_{\xi',\lambda'}^{\eta',\nu'}(x) = |\Phi_{2n}(x)|^{\lambda_{2n}'+\frac{2n-1}{2}}_{\xi_{2n}'}\prod_{j=1}^{2n-1}|\Phi_j(x)|_{\xi_j'+\eta_{2n-j}'}^{\lambda_j'-\nu_{2n-j}'-\frac{1}{2}}|\Psi_j(x)|_{\eta_{2n-j}'+\xi_{j+1}'}^{\nu_{2n-j}'-\lambda_{j+1}'-\frac{1}{2}}, $$
where
\[
\Phi_k(g)=\det(g_{ij})_{2n+1-k\leq i\leq 2n, 1\leq j\leq k}, \qquad \&\qquad \Psi_\ell(g)=\det(g_{ij})_{2n-\ell\leq i \leq 2n-1,1\leq j\leq \ell}.
\]

By Lemma \ref{lem:degseriesInsidePrincipalSeries} we can for $(\xi,\lambda,\eta,\nu)\in (\ZZ/2\ZZ)^2\times \CC^2\times (\ZZ/2\ZZ)^3\times \CC^3$ pick $(\xi',\lambda',\eta',\nu')\in (\ZZ/2\ZZ)^{2n}\times \CC^{2n}\times (\ZZ/2\ZZ)^{2n-1}\times \CC^{2n-1}$ such that $\pi_{\xi,\lambda}$ is a quotient of $\pi_{\xi',\lambda'}^{\ps}$ and $\tau_{\eta',\nu'}^{\ps}$ contains $\tau_{(\xi_2,\eta,\xi_1),(\lambda_2,\nu,\lambda_1)}$ as a subrepresentation. Then the kernel is of the form
\[ K_{\xi',\lambda'}^{\eta',\nu'}(g) = |\Phi_n(g)|^{\lambda_1-\nu-\frac{n}{2}}_{\xi_1+\eta}|\Psi_n(g)|_{\eta+\xi_2}^{\nu-\lambda_2-\frac{n}{2}}|\det(g)|^{\lambda_{2}+\frac{n}{2}}_{\xi_2}. \]
Abusing notation, we also write $K_{\xi,\lambda}^{\eta,\nu}(g)$ for this expression. Since this kernel is $P_G$-equivariant from the right and $P_H$-equivariant from the left, the corresponding intertwining operator $A_{\xi',\lambda'}^{\eta',\nu'}:\pi_{\xi',\lambda'}^{\ps}\to\tau_{\eta',\nu'}^{\ps}$ factorizes to an operator $A_{\xi,\lambda}^{\eta,\nu}:\pi_{\xi,\lambda}\to\tau_{(\xi_2,\eta,\xi_1),(\lambda_2,\nu,\lambda_1)}$ via the following diagram:
\[\begin{tikzcd}
	{\pi_{\xi',\lambda'}^{\ps}} & {\tau_{\eta',\nu'}^{\ps}} \\
	{\pi_{\xi,\lambda}} & {\tau_{(\xi_2,\eta,\xi_1),(\lambda_2,\nu,\lambda_1)}}.
	\arrow["{A_{\xi',\lambda'}^{\eta',\nu'}}", from=1-1, to=1-2]
	\arrow["{\Pi_{\xi',\lambda'}}"', two heads, from=1-1, to=2-1]
	\arrow["{A_{\xi,\lambda}^{\eta,\nu}}"', from=2-1, to=2-2]
	\arrow["{\iota_{\eta',\nu'}}"', hook, from=2-2, to=1-2]
\end{tikzcd}
\]
It follows that $A_{\xi,\lambda}^{\eta,\nu}$ is given by
\[
A_{\xi,\lambda}^{\eta,\nu}f(h)=(-1)^{(n+1)(\eta+\xi_2)}\int_{G/P_G}K_{\xi,\lambda}^{\eta,\nu}(h^{-1}g)f(g)\,d(gP_G),
\]
where we have added the sign such that in the non-compact $\overline{N}_G$-picture we have that
\[
A_{\xi,\lambda}^{\eta,\nu}f(h)=\int_{\R^{n\times n}}|\det(X)|^{\lambda_1-\nu-\frac{n}{2}}_{\xi_1+\eta}|\operatorname{det}_{n-1}(X)|_{\eta+\xi_2}^{\nu-\lambda_2-\frac{n}{2}}f(h\begin{pmatrix}
    I_n & \\
    X & I_n
\end{pmatrix})\,dX,
\]
where $\operatorname{det}_{n-1}(X)=\det(X_{ij})_{1\leq i,j\leq n-1}$. This integral converges when  $\Re(\lambda_1-\nu)>\frac{n-2}{2}$ and $\Re(\nu-\lambda_2)>\frac{n-2}{2}$ by \cite[Proposition 2.1]{DF24}.
Consider the normalization 
\[
n_\mathbf{A}(\xi,\lambda,\eta,\nu)=\Gamma(\tfrac{\lambda_1-\nu+1-\frac{n}{2}+[\xi_1+\eta]}{2})\Gamma(\tfrac{\nu-\lambda_2+1-\frac{n}{2}+[\xi_2+\eta]}{2})\prod_{i=2}^n\Gamma(\tfrac{\lambda_1-\lambda_2+i-n+[\xi_1+\xi_2]}{2})
\]
and set $\mathbf{A}_{\xi,\lambda}^{\eta,\nu}=n_\mathbf{A}(\xi,\lambda,\eta,\nu)^{-1}A_{\xi,\lambda}^{\eta,\nu}.$

\begin{proposition} \label{prop:A-SBOholomorphic}
The family $\mathbf{A}_{\xi,\lambda}^{\eta,\nu}$ depends holomorphically on $(\lambda ,\nu )\in \CC^2\times \CC$ for all $(\xi ,\eta )\in (\ZZ/2\ZZ)^2\times (\ZZ/2\ZZ)$.
\end{proposition}

We postpone the proof of this until Section \ref{sec:relating the two types of intertwining operators}. Until then we only use that $\mathbf{A}_{\xi,\lambda}^{\eta,\nu}$ depends meromorphically on $(\lambda,\nu)\in\CC^2\times\CC$.

\subsection{Second type of symmetry breaking operators}\label{subsec:B-SBO}

For $\lambda\in \CC^2$, $\xi\in (\ZZ/2\ZZ)^2$, $\nu \in \CC$ and $\eta\in \ZZ/2\ZZ$ consider the operator $B_{\xi ,\lambda }^{\eta ,\nu }:\pi _{\xi ,\lambda }\to \tau _{(\xi_1 ,\eta ,\xi_2 ),(\lambda_1 ,\nu ,\lambda_2 )}$ given by the integral
\[
B_{\xi ,\lambda }^{\eta ,\nu }f(h)=\int_{\RR}|t|_{\xi_1+\eta}^{\lambda_1 -\nu +\frac{n}{2}}f(h\overline{n}_{2n,n}(t))\frac{dt}{|t|}
\]
where $\overline{n}_{2n,n}(t)=I_{2n}+tE_{2n,n}=e^{tE_{2n,n}}$. It is straightforward computation to show that $B_{\xi ,\lambda }^{\eta ,\nu }$ indeed maps into $\tau _{(\xi_1 ,\eta ,\xi_2 ),(\lambda_1 ,\nu ,\lambda_2 )}$, and it intertwines the left action of $H$ by definition, whenever it converges. Consider the normalization
\[
n_{\mathbf{B}}(\xi ,\lambda ,\eta ,\nu )=\Gamma(\tfrac{\lambda_1 -\nu +\frac{n}{2}+[\xi_1+\eta]}{2})\Gamma(\tfrac{\nu -\lambda_2 +\frac{n}{2}+[\xi_2+\eta]}{2}),
\]
and set $\mathbf{B}_{\xi ,\lambda }^{\eta ,\nu }=n_{\mathbf{B}}(\xi ,\lambda ,\eta ,\nu )^{-1}B_{\xi ,\lambda }^{\eta ,\nu }$.

\begin{proposition} \label{prop:B-SBOholomorphic}
The integral defining $B_{\xi,\lambda}^{\eta,\nu}$ converges when $\Re(\lambda_1-\nu)>-\frac{n}{2}$ and $\Re(\nu-\lambda_2)>-\frac{n}{2}$. The family $\mathbf{B}_{\xi ,\lambda }^{\eta ,\nu }$ can be extended analytically such that it depends holomorphically on $(\lambda ,\nu )\in \CC^2\times \CC$ for all $(\xi ,\eta )\in (\ZZ/2\ZZ)^2\times (\ZZ/2\ZZ)$, and it defines a family of intertwining operators in $\Hom_H(\pi_{\xi,\lambda}|_H,\tau_{(\xi_1,\eta,\xi_2),(\lambda_1,\nu,\lambda_2)})$.
\end{proposition}

\begin{proof}
We change the integral to compact coordinates. The integral of $B_{\xi,\lambda}^{\eta,\nu}$ can be reduced to a $\SL(2,\RR)$-computation by the embedding $\iota$ of $\SL(2,\RR)$ inside $G$ at the entries $g_{2n,n}$, $g_{n,n}$, $g_{n,2n}$ and $g_{2n,2n}$. Decomposing 
\[
\begin{pmatrix}
    1 &\\
    t & 1
\end{pmatrix}=\frac{1}{\sqrt{1+t^2}}\begin{pmatrix}
    1 & -t\\
    t & 1
\end{pmatrix}\begin{pmatrix}
    \sqrt{1+t^2} & \frac{t}{\sqrt{1+t^2}}\\
    0 & \frac{1}{\sqrt{1+t^2}}
\end{pmatrix},
\]
using the equivariance of $f$ and the change of variables $t=\tan\theta$ we get 
\begin{equation}\label{eq:BinKcoordinates}
    B_{\xi,\lambda}^{\eta,\nu}f(h)=\int_{-\frac{\pi}{2}}^{\frac{\pi}{2}}|\sin\theta|^{\lambda_1-\nu+\frac{n}{2}-1}_{\xi_1+\eta}(\cos\theta)^{\nu-\lambda_2+\frac{n}{2}-1} f(h\iota(k_\theta))\,d\theta,
\end{equation}
where 
\[
k_\theta=\begin{pmatrix}
    \cos\theta & -\sin\theta\\
    \sin\theta &\cos \theta
\end{pmatrix}.
\]
Changing variables of the integral $-\frac{\pi}{2}\leq \theta\leq 0$ by $\theta\to \theta+\pi$ we get an alternate expression
\[
B_{\xi,\lambda}^{\eta,\nu}f(h)=\int_{0}^\pi (\sin\theta)^{\lambda_1-\nu+\frac{n}{2}-1}|\cos\theta|_{\eta+\xi_2}^{\lambda_2-\nu+\frac{n}{2}-1}f(h\iota(k_\theta))\,d\theta.
\]
From classical results about the Riesz distribution (see e.g. \cite{GS64}) we have that $|x|^\mu_\varepsilon/\Gamma(\frac{\mu+1+[\varepsilon]}{2})$ is locally $L^1$ when $\Re(\mu)>-1$ and it extends holomorphically to all $\mu\in\CC$.  By the substitution $x=\sin\theta$ and $x=\cos\theta$ in our two expressions we see that $B_{\xi,\lambda}^{\eta,\nu}$ indeed converges for those parameters and the normalization $n_\mathbf{B}(\xi,\lambda,\eta,\nu)$ indeed makes $\mathbf{B}_{\xi,\lambda}^{\eta,\nu}$ holomorphic.
\end{proof}

For later purpose, we prove a statement about the vanishing of $\mathbf{B}_{\xi,\lambda}^{\eta,\nu}$.

\begin{proposition}\label{prop:Bvanishing}
    If both $\lambda_1-\nu+\frac{n}{2}+[\xi_1+\eta]\in-2\NN$ and $\nu-\lambda_2+\frac{n}{2}+[\xi_2+\eta]\in-2\NN$, then $\mathbf{B}_{\xi,\lambda}^{\eta,\nu}=0$.
\end{proposition}

\begin{proof}
    Substituting $x=\sin\theta$ in \eqref{eq:BinKcoordinates} yields
    $$ B_{\xi,\lambda}^{\eta,\nu}f(h) = \int_{-1}^1|x|_{\xi_1+\eta}^{\lambda_1-\nu+\frac{n}{2}-1}(1-x^2)^{\frac{\nu-\lambda_2+\frac{n}{2}-2}{2}}f(h\iota(k_{\arcsin x}))\,dx. $$
    Write $\lambda_1-\nu+\frac{n}{2}=-m$ with $m\in\NN$, $m\equiv\xi_1+\eta\mod2$. Since $\Gamma(\frac{s+[\varepsilon]}{2})^{-1}|x|_\varepsilon^{s-1}$ at $s=-m$ is a non-zero multiple of $\delta^{(m)}(x)$, we find that
    $$ \textbf{B}_{\xi,\lambda}^{\eta,\nu}f(h) = \const\times\Gamma(\tfrac{\nu -\lambda_2 +\frac{n}{2}+[\xi_2+\eta]}{2})^{-1}\left.\frac{d^m}{dx^m}\right|_{x=0}\left[(1-x^2)^{\frac{\nu-\lambda_2+\frac{n}{2}-2}{2}}f(h\iota(k_{\arcsin x}))\right]. $$
    For $\nu-\lambda_2+\frac{n}{2}+[\xi_2+\eta]\in-2\NN$ we have $\Gamma(\tfrac{\nu -\lambda_2 +\frac{n}{2}+[\xi_2+\eta]}{2})^{-1}=0$ while the remaining terms are non-singular, so $\textbf{B}_{\xi,\lambda}^{\eta,\nu}=0$.
\end{proof}

\begin{remark}\label{rem:DiffOpResidues}
The previous proof reveals two different types of residues of $B_{\xi,\lambda}^{\eta,\nu}$ when $\lambda_1-\nu+\frac{n}{2}+[\xi_1+\eta],\nu-\lambda_2+\frac{n}{2}+[\eta+\xi_2]\in-2\NN$. If $\lambda_1-\nu+\frac{n}{2}=-m$ with $m\in\NN$, $m\equiv\xi_1+\eta$, we define
$$ \mathbf{C}_{\xi,\lambda}^{\eta,\nu}f(h) = \left.\frac{d^m}{dt^m}\right|_{t=0}f(h\overline{n}_{2n,n}(t))=\Big(\frac{\partial^m}{\partial g_{2n,n}^m}f\Big)(h), $$
and if $\nu-\lambda_2+\frac{n}{2}=-m$ with $m\in\NN$, $m\equiv\xi_2+\eta$, we define
$$ \mathbf{D}_{\xi,\lambda}^{\eta,\nu}f(h) = \left.\frac{d^m}{dt^m}\right|_{t=0}f(h\iota(k_{\frac{\pi}{2}})\overline{n}_{2n,n}(t))=\Big[\Big(\frac{\partial}{\partial g_{2n,n}}-\sum_{i=1}^{2n-1}g_{i,n}\frac{\partial}{\partial g_{i,n}}\Big)^mf\Big](h\iota(k_{\frac{\pi}{2}})). $$
Note that $\mathbf{C}_{\xi,\lambda}^{\eta,\nu}$ and $\mathbf{D}_{\xi,\lambda}^{\eta,\nu}$ are differential operators in the sense of \cite[Definition 2.1]{KP16}. While $\mathbf{C}_{\xi,\lambda}^{\eta,\nu}$ is a differential operator with respect to the standard embedding
$$H/P_H\hookrightarrow G/P_G,\quad hP_H\mapsto hP_G,$$
the operator $\mathbf{D}_{\xi,\lambda}^{\eta,\nu}$ is differential with respect to the embedding
$$H/P_H\hookrightarrow G/P_G,\quad hP_H\mapsto h\iota(k_{\frac{\pi}{2}})P_G.$$
We also refer to \cite[Section 2.4]{DL25} for a similar situation. From the proof of Proposition \ref{prop:Bvanishing} we get that if $\lambda_1-\nu+\frac{n}{2}=-m$ with $m\in\NN$, $m\equiv\xi_1+\eta$, then
    $$ \mathbf{B}_{\xi,\lambda}^{\eta,\nu} =\alpha(\xi,\lambda,\eta,\nu)\mathbf{C}_{\xi,\lambda}^{\eta,\nu}, $$
and if $\nu-\lambda_2+\frac{n}{2}=-m$ with $m\in\NN$, $m\equiv\xi_2+\eta$, then
    $$ \mathbf{B}_{\xi,\lambda}^{\eta,\nu} = \alpha'(\xi,\lambda,\eta,\nu)\mathbf{D}_{\xi,\lambda}^{\eta,\nu}, $$
where $\alpha$, $\alpha'$ are holomorphic functions in $(\lambda,\nu)\in \CC^2\times \CC$ for every $(\xi,\eta)\in (\ZZ/2\ZZ)^2\times (\ZZ/2\ZZ)$. However, these residues are not relevant for the unitary branching problem considered in this paper.
\end{remark}

\begin{proposition} \label{prop:B-surjective}
For $\lambda_1-\nu+\frac{n}{2}+[\xi_1+\eta],\nu-\lambda_2+\frac{n}{2}+[\xi_2+\eta]\not\in-2\NN$ the operator ${\mathbf B}_{\xi,\lambda}^{\eta,\nu}$ is surjective.
\end{proposition}

We postpone the proof to Section~\ref{sec:Plancherel} where we find a different expression for $B_{\xi,\lambda}^{\eta,\nu}$.

\subsection{The open \texorpdfstring{$H$}{H}-orbit in \texorpdfstring{$G/P_G$}{G/PG} and an integral formula}

The action of $H$ on $G/P_G$ has an open dense orbit:

\begin{lemma}
    $H$ acts on $G/P_G$ with a unique open orbit, the orbit through the point
    \[
        x_0=I_{2n}+E_{2n,n}=\overline{n}_{2n,n}(1).
    \]
    In particular, this orbit is dense.
\end{lemma}

\begin{proof}
    We identify $G/P_G$ with the Grassmannian of $n$-dimensional subspaces of $\RR^{2n}$ and claim that
    $$ \calO = \{V\in G/P_G:\RR e_{2n}\not\subseteq V\not\subseteq\RR^{2n-1}\times\{0\}\} $$
    is a single $H$-orbit. Since this subset is clearly open and dense, and $x_0P_G$ corresponds to the subspace $\RR e_1\oplus\cdots\oplus\RR e_{n-1}\oplus\RR(e_n+e_{2n})\in\calO$, the claim follows. To show that $\calO$ is an $H$-orbit let $V\in\calO$ and write $V=\RR v_1\oplus\cdots\oplus\RR v_n$ with $v_1,\ldots,v_n\in\RR^{2n}$. Since $V\not\subseteq\RR^{2n-1}\times\{0\}$, one of the vectors $v_1,\ldots,v_n$ has to have a non-trivial $2n$-th component, say $v_n$. By rescaling we may assume that $v_n=v_n'+e_{2n}$ with $v_n'\in\RR^{2n-1}\times\{0\}$. Moreover, replacing each $v_j$, $j=1,\ldots,n-1$, by a linear combination of $v_j$ and $v_n$ we may assume that $v_1,\ldots,v_{n-1}\in\RR^{2n-1}\times\{0\}$. Since $\RR e_{2n}\not\subseteq V$, the vectors $v_1,\ldots,v_{n-1},v_n'\in\RR^{2n-1}\times\{0\}$ must be linearly independent. Hence, there exists $h\in H$ such that $hv_j=e_j$ for $j=1,\ldots,n-1$ and $hv_n'=e_n$. This shows the claim.
\end{proof}
The stabilizer subgroup $S$ of $x_0$ in $H$ is given by 
\[
S=\begin{pmatrix}
    \GL(n-1,\RR) & \RR^{n-1} & \RR^{(n-1)\times (n-1)}\\
    0 & 1 & \RR^{n-1}\\
    0 & 0 &\GL(n-1,\RR)
\end{pmatrix}.
\]

\begin{corollary}\label{cor:integralformula}
For $f\in C(G\times_{P_G}\CC_{2\rho_G})$ we have the following integral formula
\[
\int_{G/P_G}f(g)\,d(gP_G)=\int_{H/P_H}\bigg(\int_{\RR^\times} |t|^nf(h\overline{n}_{2n,n}(t))\,\frac{dt}{|t|}\bigg)d(hP_H),
\]
\end{corollary}

\begin{proof}
The density bundle $G\times_{P_G}\CC_{2\rho_G}$ over $G/P_G$ restricts to the density bundle over the open $H$-orbit $H/S$. Denote by $d(hS)$ the unique $H$-invariant integral on the latter. Since the open $H$-orbit is dense, we can normalize $d(hS)$ so that
\[
\int_{G/P_G}f(g)\,d(gP_G)=\int_{H/S}f(hx_0)\,d(hS)=\int_{H/P_H}\int_{P_H/S}f(hpx_0)\,d(pS)\,d(hP_H),
\]
where we have used \cite[Theorem 1.1]{CF23} in the second step. We have that $P_H/S\simeq \RR^\times $ by $t\mapsto \diag(I_{n-1},t,I_{n-1})S$, so the result follows from the identity $\diag(I_{n-1},t,I_{n-1})x_0=\overline{n}_{2n,n}(t^{-1})\diag(I_{n-1},t,I_{n-1})$ and the change of variables $t\to t^{-1}$.
\end{proof}

\subsection{Relating the two types of intertwining operators}\label{sec:relating the two types of intertwining operators}

In this section we show how the two types of symmetry breaking operators are related by the Knapp--Stein intertwining operators. We use these relations to finish the proof that $\mathbf{A}_{\xi,\lambda}^{\eta,\nu}$ is holomorphic in $(\lambda,\nu)\in\CC^2\times\CC^3$ (see Proposition~\ref{prop:A-SBOholomorphic}).

Using the functions $\Phi_i$ and $\Psi_j$ from Subsection \ref{subsec:A-SBO} we can rewrite the Knapp--Stein intertwining operators as integrals over $G/P_G$ and $H/P_H$ in the following way
\[
T_{\xi,\lambda}f(g_0)=\int_{G/P_G}|\Phi_n(g)|^{\lambda_1-\lambda_2-n}_{\xi_1+\xi_2}|\det(g)|^{\lambda_2+\frac{n}{2}}_{\xi_2} f(g_0g)\,d(gP_G)
\]
and 
\[
S_{\eta,\nu}f(h_0)=(-1)^{(n+1)(\xi_2+\eta)}\!\!\!\int_{H/P_G}\! \!|\Psi_{n-1}(h)|^{\nu_1-\nu_2-\frac{n}{2}}_{\eta_1+\eta_2}|\Psi_n(h)|^{\nu_2-\nu_3-\frac{n}{2}}_{\eta_2+\eta_3}|\det(h)|^{\nu_3+\frac{n}{2}}_{\eta_3}f(h_0h)\,d(hP_H)
\]
This can be verified by seeing that the values of the integral kernels coincide with integral kernels for the Knapp--Stein intertwining operators obtained in Section \ref{sec:KnappSteinAndUnitary} on $\overline{N}_G$ and $\overline{N}_H$ respectively. 
In total we have defined four different families of operators namely $\mathbf{T}_{\xi,\lambda}$, $\mathbf{S}_{\eta,\nu}$, $\mathbf{A}_{\xi,\lambda}^{\eta,\nu}$ and $\mathbf{B}_{\xi,\lambda}^{\eta,\nu}$. The following proposition shows how these are related. The statement can be visualized by the following diagram which commutes up to a constant made explicit in the proposition.
\[\begin{tikzcd}
	{\pi_{(\xi_1,\xi_2),(\lambda_1,\lambda_2)}} &&& {\pi_{(\xi_2,\xi_1),(\lambda_2,\lambda_1)}}\\ &&& \\
	{\tau_{(\xi_1,\nu,\xi_2),(\lambda_1,\nu,\lambda_2)}} &&& {\tau_{(\xi_2,\eta,\xi_1),(\lambda_2,\nu,\lambda_1)}}
	\arrow["{\mathbf{B}_{\xi,\lambda}^{\eta,\nu}}"', from=1-1, to=3-1]
	\arrow["{\mathbf{T}_{\xi,\lambda}}", from=1-1, to=1-4]
	\arrow["{\mathbf{A}_{\xi,\lambda}^{\eta,\nu}}"{description}, from=1-1, to=3-4]
	\arrow["{\mathbf{S}_{(\xi_1,\eta,\xi_2),(\lambda_1,\nu,\lambda_2)}}"', from=3-1, to=3-4]
	\arrow["{\mathbf{B}_{(\xi_2,\xi_1),(\lambda_2,\lambda_1)}^{\eta,\nu}}", from=1-4, to=3-4]
\end{tikzcd}\]

\begin{proposition}\label{Prop: functional equations}
We have the following relations between $ \mathbf{A}_{\xi ,\lambda }^{\eta ,\nu }$ and $\mathbf{B}_{\xi ,\lambda }^{\eta ,\nu }$:
\begin{equation} \label{eq:Functionaleq: TB=A}
\mathbf{S}_{(\xi_1,\eta,\xi_2) ,(\lambda_1,\nu,\lambda_2) } \circ \mathbf{B}_{\xi ,\lambda }^{\eta ,\nu }=\frac{1}{n_\mathbf{B}(\xi,\lambda,\eta,\nu)} \mathbf{A}_{\xi ,\lambda }^{\eta ,\nu }
\end{equation}
and
\begin{equation}\label{eq:Functionaleq: BT=A}
\mathbf{B}_{w_G (\xi ,\lambda )}^{\eta ,\nu }\circ \mathbf{T} _{\xi ,\lambda }=\frac{\sqrt{\pi}(-1)^{(\eta+\xi_2)(\xi_1+\xi_2)}}{\Gamma(\frac{\lambda_2-\lambda_1+n+[\xi_2+\xi_1]}{2})} \mathbf{A}_{\xi ,\lambda }^{\eta ,\nu }.
\end{equation}
\end{proposition}

\begin{proof}
For the first identity, we write 
\[
A_{\xi,\lambda}^{\eta,\nu}f(h_0)=\int_{G/P_G}|\Phi_n(g)|^{\lambda_1-\nu-\frac{n}{2}}_{\xi_1+\eta}|\Psi_n(g)|^{\nu-\lambda_2-\frac{n}{2}}_{\eta+\xi_2}|\det(g)|^{\lambda_2+\frac{n}{2}}_{\xi_2}f(h_0g)\,d(gP_G)
\]
and use Corollary \ref{cor:integralformula} to rewrite this expression as
\[
\int_{H/P_H}|\Psi_{n-1}(h)|^{\lambda_1-\nu-\frac{n}{2}}_{\xi_1+\eta}|\Psi_n(h)|^{\nu-\lambda_2-\frac{n}{2}}_{\eta+\xi_2}|\det(h)|^{\lambda_2+\frac{n}{2}}_{\xi_2}\int_{\RR^\times} |t|^{\lambda_1-\nu+\frac{n}{2}}_{\xi_1+\eta}f(h_0h\overline{n}_{2n,n}(t))\,\frac{dt}{|t|}\,d(hP_H),
\]
since $\Phi_n(h\overline{n}_{2n,n}(t))=t\Psi_{n-1}(h)$. This is the first identity up to normalizations.

For the second identity write
\begin{multline*}
(B_{w_G(\xi,\lambda)}^{\eta,\nu}\circ T_{\xi,\lambda})f(h_0)\\=\int_{\RR^\times}|t|^{\lambda_2-\nu+\frac{n}{2}}_{\xi_2+\eta}\int_{G/P_G}|\Phi_n(g)|^{\lambda_1-\lambda_2-n}_{\xi_1+\xi_2}|\det(g)|^{\lambda_2+\frac{n}{2}}_{\xi_2}f(h_0\overline{n}_{2n,n}(t)g)\,d(gP_G)\,\frac{dt}{|t|}
\end{multline*}
and by using the $G$-invariance of the integral $d(gP_G)$ we get 
\[
\int_{G/P_G}|\det(g)|^{\lambda_2+\frac{n}{2}}_{\xi_2}f(h_0g)\int_{\RR^\times} |t|^{\lambda_2-\nu+\frac{n}{2}}_{\xi_2+\eta}|\Phi_n(g)+(-1)^nt\Psi_n(g)|^{\lambda_1-\lambda_2-n}_{\xi_1+\xi_2}\frac{dt}{|t|}\,d(gP_G),
\]
since $\Phi_n(\overline{n}_{2n,n}(-t)g)=\Phi_n(g)+(-1)^nt\Psi_n(g)$. The inner integral can be computed using \cite[Proposition A.1]{DF24} giving us $\sqrt{\pi}(-1)^{(\eta+\xi_2)(\xi_1+\xi_2)}$, some gamma-factors and the factors $|\Phi_n(g)|^{\lambda_1-\nu-\frac{n}{2}}_{\xi_1+\eta}$ and $|\Psi_n(g)|^{\nu-\lambda_2-\frac{n}{2}}_{\eta+\xi_2}$ which is then exactly the integral kernel of $A_{\xi,\lambda}^{\eta,\nu}$.
\end{proof}

We have already established that $\mathbf{T}_{\xi,\lambda}$, $\mathbf{S}_{\eta,\nu}$ and $\mathbf{B}_{\xi,\lambda}^{\eta,\nu}$ are holomorphic, so we can use these identities to show that $\mathbf{A}_{\xi,\lambda}^{\eta,\nu}$ is holomorphic as well.

\begin{proof}[Proof of Proposition~\ref{prop:A-SBOholomorphic}]
Applying \eqref{eq:Functionaleq: BT=A} in combination with Propositions~\ref{prop:storTonyholomorphic} and \ref{prop:B-SBOholomorphic} we see that $\Gamma(\frac{\lambda_2-\lambda_1+n+[\xi_1+\xi_2]}{2})^{-1}\mathbf{A}_{\xi,\lambda}^{\eta,\nu}$ is holomorphic. The same argument with \eqref{eq:Functionaleq: TB=A} in combination with Propositions~\ref{prop:lilleTonyholomorphic} and \ref{prop:B-SBOholomorphic} shows that $n_\mathbf{B}(\xi,\lambda,\eta,\nu)^{-1}\mathbf{A}_{\xi ,\lambda }^{\eta ,\nu }$ is holomorphic. Since the intersection of the poles of $\Gamma(\frac{\lambda_2-\lambda_1+n+[\xi_1+\xi_2]}{2})$ and $n_\mathbf{B}(\xi,\lambda,\eta,\nu)$ is of codimension at least two, the claim follows from Hartog's Theorem.
\end{proof}

\section{A Plancherel formula for the unitary principal series}\label{sec:Plancherel}

For $\lambda\in i\RR^2$, the representation $\pi_{\xi,\lambda}$ is unitary with respect to the $L^2$-inner product on sections of the corresponding line bundle over $G/P_G$. Therefore, the decomposition of $\pi_{\xi,\lambda}|_H$ is related to the open orbits of $H$ on $G/P_G$.

The stabilizer subgroup $S$ is contained in the parabolic subgroup $P_H$ and $H/S\to H/P_H$ is a principal fiber bundle with fiber $\RR^\times$. We consider the decomposition $S=M_SA_SN_S$ given by
\[
M_S=\begin{pmatrix}
    \SL^{\pm}(n-1,\RR)& &\\
    & 1 &\\
    & & \SL^{\pm}(n-1,\RR)
\end{pmatrix}\qquad \&\qquad A_S=\begin{pmatrix}
    \RR_+ I_{n-1} & &\\
    & 1 & \\
     & & \RR_+ I_{n-1}
\end{pmatrix}
\]
and $N_S=N_H$. In a similar way to Section \ref{sec:Representations} we can consider characters of $M_S$ as $\xi=(\xi_1,\xi_2)\in (\ZZ/2\ZZ)^2$ and of $A_S$ as $\lambda=(\lambda_1,\lambda_2)\in \CC^2$ and extend $\xi\otimes e^{\lambda}$ trivially to $S=M_SA_SN_S$ and write $\xi\otimes e^\lambda\otimes 1$ for it. Using smooth (normalized) induction from $S$ to $H$ we obtain $\Ind_S^H(\xi\otimes e^\lambda\otimes1)$ as the left-regular representation on 
\[
\{f\in C^\infty(H)\,|\,f(hman)=\xi(m)^{-1}a^{-\lambda-\rho_G}f(h)\,\forall man\in M_SA_SN_S\}.
\]

\subsection{The Mellin transform}

For $\xi \in (\ZZ/2\ZZ)^2$ and $\lambda \in \CC^2$ we can then consider the following intertwining map 
\[
\rest:\pi_{\xi,\lambda}\to \Ind_S^H(\xi\otimes e^\lambda\otimes 1),\qquad f\mapsto\big[h\mapsto f(hx_0)\big],
\]
which is an isometric isomorphism on $L^2$-sections for $\lambda\in i\RR^2$. Denote by $L^2\mbox{-}\Ind_S^H(\xi\otimes e^\lambda\otimes 1)$ the unitarily induced representation in case $\lambda\in i\RR^2$ and consider the map $\mathcal{M}_{\eta,\nu}:L^2\mbox{-}\Ind_S^H(\xi\otimes e^\lambda\otimes 1)\to\tau_{(\xi_1,\eta,\xi_2),(\lambda_1,\nu,\lambda_2)}$ given by
\[
\mathcal{M}_{\eta,\nu}f(h)=\frac{1}{2\sqrt{\pi}}\int_{\RR^\times}|t|^{\nu}_{\eta} f\Big(h\diag(I_{n-1},t,I_{n-1}) \Big)\frac{dt}{|t|}.
\]

\begin{proposition}
For $\lambda\in i\RR^2$ the map $f\mapsto(\mathcal{M}_{\eta,\nu}f)_{\eta\in \{0,1\},\,\nu\in i\RR}$ defines a unitary isomorphism
\[
L^2\mbox{-}\Ind_S^H(\xi\otimes e^\lambda\otimes 1)\to \bigoplus_{\eta=0}^1\int_{i\RR}^\oplus L^2\mbox{-}\Ind_{P_H}^H((\xi_1,\eta,\xi_2)\otimes e^{(\lambda_1,\nu,\lambda_2)}\otimes 1)\,d\nu.
\]
\end{proposition}
\begin{proof}
The map $\mathcal{M}_{\eta,\nu}$ is the Mellin transform of the even- and odd parts of
$$ f_h(t):= f(h\diag(I_{n-1},t,I_{n-1})) \qquad (t\in\RR^\times). $$
Since $P_H/S\simeq \RR^\times$, by \cite[Theorem 1.1]{CF23},
\begin{align*}
\int_{H/S}|f(h)|^2\,d(hS)&=\int_{H/P_H}\int_{P_H/S}|f(hp)|^2\,d(pS)\,d(hP_H)\\
&=\int_{H/P_H}\| f_h\|_{L^2(\RR^\times,\frac{dt}{|t|})}\,d(hP_H)\\
&=\sum_{\eta=0}^1\int_{i\RR}\| \mathcal{M}_{\eta,\nu}f\|^2_{L^2(H/P_H)}\,d\nu,
\end{align*}
where we use that the Mellin transform is a unitary isomorphism from $L^2(\RR_{>0},\frac{dt}{t})$ to $L^2(\RR,ds)$ on the even and odd part of $f_h$.
\end{proof}

\begin{proposition}
We have the following relation
\[
\mathcal{M}_{\eta,\nu}\circ \rest=\frac{1}{2\sqrt{\pi}}B_{\xi,\lambda}^{\eta,\nu}.
\]
\end{proposition}

\begin{proof}
Note that
$$ \diag(I_{n-1},t,I_n)x_0 = \overline{n}_{2n,n}(t^{-1})\diag(I_{n-1},t,I_n), $$
thus the identity follows from the equivariance of $f$ and the substitution $t\to t^{-1}$ in the integral. 
\end{proof}

\begin{corollary}\label{cor:UnitaryPlancherel}
For $\lambda\in i\RR^2$ and $f\in \pi_{\xi,\lambda}$ we have the following Plancherel formula
\begin{equation}\label{eq: Plancherel formula}
\| f\|^2_{L^2(K_{2n}/M_G)}=\frac{1}{4\pi}\sum_{\eta\in\ZZ/2\ZZ}\int_{i\RR} \|B_{\xi,\lambda}^{\eta,\nu}f\|^2_{L^2(K_{2n-1}/M_H)}\,d\nu,
\end{equation}
and the corresponding direct integral decomposition
\[
\pi_{\xi,\lambda}|_H\simeq \bigoplus_{\eta\in\ZZ/2\ZZ}\int_{i\RR_+}^\oplus \tau_{(\xi_1,\eta,\xi_2),(\lambda_1,\nu,\lambda_2)}\,d\nu.
\]
\end{corollary}

With this way of expressing $B_{\xi,\lambda}^{\eta,\nu}$ we can now give a proof of Proposition~\ref{prop:B-surjective}.

\begin{proof}[Proof of Proposition~\ref{prop:B-surjective}]
    The assumption assures that the normalizing factor $n_{\mathbf B}(\xi,\lambda,\eta,\nu)$ is non-singular, so the claim is equivalent to the unnormalized operator $B_{\xi,\lambda}^{\eta,\nu}$ being surjective. We can therefore use the previous proposition. The restriction map $$\rest:\pi_{\xi,\lambda}\to \Ind_S^H(\xi\otimes e^\lambda\otimes 1)$$ clearly has all $C_c^\infty$ sections on $H/S$ in its image, since those can simply be extended by zero to smooth sections on $G/P_G$. It therefore suffices to show that the Mellin transform $$\mathcal{M}_{\eta,\nu}:C_c^\infty\mbox{-}\Ind_S^H(\xi\otimes e^\lambda\otimes 1)\to\tau_{(\xi_1,\eta,\xi_2),(\lambda_1,\nu,\lambda_2)}$$ maps $C_c^\infty$ sections onto $\tau_{(\xi_1,\eta,\xi_2),(\lambda_1,\nu,\lambda_2)}$. By trivializing the fibration $H/S\to H/P_H$ and using a partition of unity this is reduced to the following local statement: Let $\Omega\subseteq\RR^d$ be open and $\eta\in\ZZ/2\ZZ$, $\nu\in\CC$, then the following operator is surjective:
    $$ \widetilde{\mathcal M}_{\eta,\nu}: C_c^\infty(\Omega\times\RR^\times)\to C_c^\infty(\Omega), \quad \widetilde{\mathcal M}_{\eta,\nu}f(x) = \int_{\RR^\times}|t|_\eta^\nu f(x,t)\frac{dt}{|t|}. $$
    This statement follows by considering functions of the form $f(x,t)=g(x)\varphi(t)$, where $\varphi\in C_c^\infty(\RR^\times)$ is a fixed test function (depending on $\eta$ and $\nu$) with $\int_{\RR^\times}|t|^\nu_\eta\varphi(t)\frac{dt}{|t|}\neq0$.
\end{proof}

\section{Decomposing unitary representations}\label{sec:DecomposingUnitaryRepresentations}

We now decompose the restriction to $H$ of all unitary representations of the degenerate series $\pi_{\xi,\lambda}$ and their unitarizable quotients.

\subsection{A Plancherel formula for the invariant Hermitian form}

Recall the invariant pairings $(\cdot\,|\,\cdot)_G$ and $(\cdot\,|\,\cdot)_H$ from \eqref{eq:InvPairingG} and \eqref{eq:InvPairingH}. When $\lambda\in i\RR^2$ the Hermitian form $(\,\cdot \,|\,\overline{\cdot}\,)_G$ on $\pi_{\xi,\lambda}\times \pi_{\xi,\lambda}$ defines the $L^2(K_{2n}/M_G)$-inner product and similarly for $\nu \in i\RR^3$ the Hermitian form $(\,\cdot \,|\,\overline{\cdot}\,)_H$ defines the $L^2(K_{2n-1}/M_H)$-inner product. Thus we can write the Plancherel formula \eqref{eq: Plancherel formula} for $f\in \pi_{\xi,\lambda}$ and $f'\in \pi_{\xi,-\lambda}$ as
\[
 (f\,|\,f')_G=(f\,|\,\overline{\overline{f'}})_G=\frac{1}{4\pi}\sum_{\eta=0}^1\int_{i\RR} ( B_{\xi,\lambda}^{\eta,\nu}f\,|\,\overline{B_{\xi,\lambda}^{\eta,\nu}\overline{f'}})_H\,d\nu=\frac{1}{4\pi}\sum_{\eta=0}^1\int_{i\RR} ( B_{\xi,\lambda}^{\eta,\nu}f\,|\,B_{\xi,-\lambda}^{\eta,-\nu}f')_H\,d\nu.
\]
Using this for $\xi_1=\xi_2$ and $\lambda\in i\RR^2$ with $f\in \pi_{\xi,\lambda}$ and $f'\in \pi_{\xi,-w_G\lambda}$ we get, using the functional equation of Proposition \ref{Prop: functional equations}, that
\begin{align}
    ( f\,|\, \mathbf{T}_{\xi,-w_G\lambda}& f')_G =\frac{1}{4\pi}\sum_{\eta=0}^1\int_{i\RR}(B_{\xi,\lambda}^{\eta,\nu}f\,|\,B_{\xi,-\lambda}^{\eta,-\nu}\circ \mathbf{T}_{\xi,-w_G\lambda}f')_H\,d\nu \label{eq: Plancherel på unitær akse} \\
    &=\frac{1}{4\sqrt{\pi}\Gamma(\frac{\lambda_2-\lambda_1+n}{2})}\sum_{\eta=0}^1\int_{i\RR} n_\mathbf{B}(\xi,\lambda,\eta,\nu)n_\mathbf{B}(\xi,-\lambda,\eta,-\nu)(\mathbf{B}_{\xi,\lambda}^{\eta,\nu}f\,|\,\mathbf{A}_{\xi,-w_G\lambda}^{\eta,-\nu}f')_H\,d\nu.\nonumber 
\end{align}
The left hand side is holomorphic in $\lambda\in \CC^2$ and when $f'=\overline{f}$ it induces the invariant norm on the unitarizable quotients of $\pi_{\xi,\lambda}$.  We want to analytically continue the right hand side from $\lambda\in i\RR^2$ to those $\lambda$ for which there are unitarizable quotients of $\pi_{\xi,\lambda}$, since by the identity theorem for holomorphic functions it remains equal to the left hand side. This will give us various branching laws. 

\subsection{Analytic continuation}\label{sec:AnalyticContinuation}
Set $\xi_1=\xi_2$ and $\lambda_1+\lambda_2=0$ and let $\mu=\lambda_1=-\lambda_2$ and $\delta=\xi_1+\eta=\xi_2+\eta$. Consider the meromorphic function
\begin{multline*}
    a(\delta,\mu,\nu) := n_\mathbf{B}(\xi,\lambda,\eta,\nu)n_\mathbf{B}(\xi,-\lambda,\eta,-\nu)\\
    = \Gamma\left(\frac{\mu+\nu+\frac{n}{2}+\delta}{2}\right)\Gamma\left(\frac{\mu-\nu+\frac{n}{2}+\delta}{2}\right)\Gamma\left(\frac{-\mu+\nu+\frac{n}{2}+\delta}{2}\right)\Gamma\left(\frac{-\mu-\nu+\frac{n}{2}+\delta}{2}\right),
\end{multline*}
and the holomorphic function
\[
Q(\delta,\mu,\nu)=(\mathbf{B}_{\xi,\lambda}^{\eta,\nu}f\,|\,\mathbf{A}_{\xi,-w_2\lambda}^{\eta,-\nu}f')_H,
\]
then \eqref{eq: Plancherel på unitær akse} can be written as
\begin{equation}
    ( f\, |\,\mathbf{T}_{\xi,-w_G\lambda}f')_G = \frac{1}{4\sqrt{\pi}\Gamma(\frac{n}{2}-\mu)}\sum_{\delta=0}^1\int_{i\RR}a(\delta,\mu,\nu)Q(\delta,\mu,\nu)\,d\nu.\label{eq:PlancherelWithAandQ}
\end{equation}

\begin{proposition} \label{prop:AnalyticCont}
    For all $\nu\in i\RR$ the function $a(\delta,\mu,\nu)Q(\delta,\mu,\nu)$ is holomorphic in the half-plane $\{\mu\in\CC:\Re\mu<\frac{n}{2}\}$.
\end{proposition}

\begin{proof}
    The poles of $\mu\mapsto a(\delta,\mu,\nu)$ in the half-plane $\{\mu\in\CC:\Re\mu<\frac{n}{2}\}$ counted with multiplicity are given by
    $$ \mu\in\nu-\tfrac{n}{2}-[\delta]-2\NN \qquad \mbox{and} \qquad \mu\in-\nu-\tfrac{n}{2}-[\delta]-2\NN. $$
    Therefore, the statement follows from the following lemma.
\end{proof}

\begin{lemma}\label{lem:VanishingOfQ}
    \begin{enumerate}
        \item\label{lem:VanishingOfQ1} For $\nu\in i\RR\setminus\{0\}$ the function $\mu\mapsto Q(\delta,\mu,\nu)$ vanishes at $\mu\in\pm\nu-\frac{n}{2}-[\delta]-2\NN$.
        \item\label{lem:VanishingOfQ2} For $\nu=0$ the function $\mu\mapsto Q(\delta,\mu,\nu)$ has a zero of order at least two at $\mu\in-\frac{n}{2}-[\delta]-2\NN$.
    \end{enumerate}
\end{lemma}

\begin{proof}
    For arbitrary $k,\ell\in\NN$ we write $(z_1,z_2)=(\mu-\nu+\frac{n}{2}+[\delta]+2k,\mu+\nu+\frac{n}{2}+[\delta]+2\ell)$ and expand $\mathbf{B}_{\xi,\lambda}^{\eta,\nu}$ around $(z_1,z_2)=(0,0)$ as
    $$ \mathbf{B}_{\xi,\lambda}^{\eta,\nu} = \sum_{i,j=0}^\infty z_1^i z_2^j B_{i,j}, $$
    where $B_{i,j}\in\Hom(C^\infty(K_{2n}\times_{M_G\cap K_{2n}}\xi),C^\infty(K_{2n-1}\times_{M_H\cap K_{2n-1}}\eta))$ (in the compact picture). By Proposition~\ref{prop:Bvanishing} we have $B_{0,0}=0$. Moreover, applying \eqref{eq:Functionaleq: TB=A} yields
    \begin{multline*}
        \sum_{i,j=0}^\infty z_1^iz_2^j\cdot\mathbf{S}_{(\delta,\eta,\delta),(\mu,\nu,-\mu)} \circ B_{i,j} = \mathbf{S}_{(\delta,\eta,\delta),(\mu,\nu,-\mu)}\circ \mathbf{B}_{\xi,\lambda}^{\eta,\nu}=\frac{1}{n_\mathbf{B}(\xi,\lambda,\eta,\nu)} \mathbf{A}_{\xi,\lambda}^{\eta,\nu}\\
        = \frac{1}{\Gamma(\frac{z_1}{2}-k)\Gamma(\frac{z_2}{2}-\ell)} \mathbf{A}_{\xi,\lambda}^{\eta,\nu},
    \end{multline*}
    so putting $z_1=0$ resp. $z_2=0$ yields $\mathbf{S}_{(\delta,\eta,\delta),(\mu,\nu,-\mu)}\circ B_{0,j}=0$ for all $j\in\NN$ resp. $\mathbf{S}_{(\delta,\eta,\delta),(\mu,\nu,-\mu)}\circ B_{i,0}=0$ for all $i\in\NN$. This allows us to rewrite the identity as
    $$ \mathbf{A}_{\xi,\lambda}^{\eta,\nu} = \mathbf{S}_{(\delta,\eta,\delta),(\mu,\nu,-\mu)}\circ\left(z_1\Gamma(\tfrac{z_1}{2}-k)\cdot z_2\Gamma(\tfrac{z_2}{2}-\ell)\sum_{i,j=1}^\infty z_1^{i-1}z_2^{j-1}\cdot B_{i,j}\right). $$
    Note that the expression in parentheses is holomorphic at $(z_1,z_2)=(0,0)$. The same argument works for $-\nu$ instead of $\nu$ if we interchange $z_1$ and $z_2$ as well as $k$ and $\ell$, so we can write
    $$ \mathbf{A}_{\xi,\lambda}^{\eta,-\nu} = \mathbf{S}_{(\delta,\eta,\delta),(\mu,-\nu,-\mu)}\circ\widetilde{\mathbf{B}}_{z_1,z_2}, $$
    where $\widetilde{\mathbf{B}}_{z_1,z_2}$ is holomorphic in $(z_1,z_2)=(0,0)$. Using \eqref{eq:TonySymmetricH}, it follows that
    \begin{multline*}
    Q(\delta,\mu,\nu) = (\mathbf{B}_{\xi,\lambda}^{\eta,\nu}f\,|\,\mathbf{A}_{\xi,\lambda}^{\eta,-\nu}f)_H = \sum_{i,j=0}^\infty z_1^i z_2^j\,( B_{i,j}\,|\,\mathbf{S}_{(\delta,\eta,\delta),(\mu,\nu,-\mu)}\circ\widetilde{\mathbf{B}}_{z_1,z_2})_H\\
    = \sum_{i,j=0}^\infty z_1^i z_2^j\,( \mathbf{S}_{(\delta,\eta,\delta),(\mu,\nu,-\mu)}\circ B_{i,j}\,|\,\widetilde{\mathbf{B}}_{z_1,z_2})_H,
    \end{multline*}
    and since $\mathbf{S}_{(\delta,\eta,\delta),(\mu,\nu,-\mu)} \circ B_{i,0}=\mathbf{S}_{(\delta,\eta,\delta),(\mu,\nu,-\mu)} \circ B_{0,j}=0$ for all $i,j\in\NN$ this can be simplified to
    $$ Q(\delta,\mu,\nu) = \sum_{i,j=1}^\infty z_1^i z_2^j\,( \mathbf{S}_{(\delta,\eta,\delta),(\mu,\nu,-\mu)}\circ B_{i,j}\,|\,\widetilde{\mathbf{B}}_{z_1,z_2})_H. $$
    Now, to prove \eqref{lem:VanishingOfQ1} we observe that $\mu=\nu-\frac{n}{2}-[\delta]-2k$ resp. $\mu=-\nu-\frac{n}{2}-[\delta]-2\ell$ corresponds to $z_1=0$ resp. $z_2=0$. But since every term in the sum above contains at least one power of $z_1$ and $z_2$, we find $Q(\delta,\mu,\nu)=0$ in these cases. To show \eqref{lem:VanishingOfQ2} we put $k=\ell$, then $\nu=0$ corresponds to $z_1=z_2$. But for $z_1=z_2=\mu+\frac{n}{2}+[\delta]+2k$ we have
    $$ Q(\delta,\mu,\nu) = \sum_{i,j=1}^\infty(\mu+\tfrac{n}{2}+[\delta]+2k)^{i+j}\,( \mathbf{S}_{(\delta,\eta,\delta),(\mu,\nu,-\mu)}\circ B_{i,j}\,|\,\widetilde{\mathbf{B}}_{z_1,z_2})_H $$
    which has a zero of order at least two at $\mu=-\frac{n}{2}-[\delta]-2k$.
\end{proof}

In order to conclude that the right hand side of \eqref{eq:PlancherelWithAandQ} is holomorphic in $\mu$ in the right half-plane $\{\mu\in\CC:\Re\mu<\frac{n}{2}\}$, we need to ensure that the integrals converge for all such parameters. For this, we show that $a(\delta,\mu,\nu)Q(\delta,\mu,\nu)$ is rapidly decreasing along $\nu\in i\RR$.

\begin{theorem}\label{thm:DecayOfBinNu}
    Fix $\xi\in(\ZZ/2\ZZ)^2$, $\eta\in\ZZ/2\ZZ$ and $\lambda\in\CC^2$. For every $R>0$ and $N\in\NN$ there exists a continuous seminorm $q$ on $\pi_{\xi,\lambda}$ such that
    $$ \|B_{\xi,\lambda}^{\eta,\nu}f\|_{L^\infty(K_{2n-1})} \leq q(f)(1+|\Im\nu|)^{-N} $$
    for all $f\in\pi_{\xi,\lambda}$ and $\nu\in\CC$ with $|\Re\nu|\leq R$ and $|\Im\nu|$ sufficiently large.
\end{theorem}

The proof is rather technical and has for the sake of readability been moved to Appendix~\ref{app:ProofOfEstimate}.

\begin{corollary}
    Fix $\mu\in\RR$ and $\delta\in\ZZ/2\ZZ$. For every $N\in\NN$ there exists a constant $C>0$ such that
    \begin{equation}
        \left|a(\delta,\mu,\nu)Q(\delta,\mu,\nu)\right| \leq C(1+|\nu|)^{-N} \qquad \mbox{for all }\nu\in i\RR.\label{eq:DecayOfaQinNu}
    \end{equation}
\end{corollary}

\begin{proof}
    First note that using \eqref{eq:Functionaleq: BT=A} the statement of Theorem~\ref{thm:DecayOfBinNu} holds in the same way for $A_{\xi,\lambda}^{\eta,\nu}$. It follows that \eqref{eq:DecayOfaQinNu} holds for the pairing $(B_{\xi,\lambda}^{\eta,\nu}f\,|\,A_{\xi,-w_G\lambda}^{\eta,-\nu}f')_H$ and $|\Im\nu|$ sufficiently large. Since the normalization of $B_{\xi,\lambda}^{\eta,\nu}$ and $A_{\xi,-w_G\lambda}^{\eta,-\nu}$ is (up to a constant depending only on $\mu$) the same as $a(\delta,\mu,\nu)$, the estimate \eqref{eq:DecayOfaQinNu} follows for $|\Im\nu|$ sufficiently large. But on a bounded set of $\nu\in i\RR$, the left hand side of \eqref{eq:DecayOfaQinNu} is continuous and hence bounded above, and the right hand side is bounded below by a positive constant. Hence, by choosing $C>0$ sufficiently large, we can drop the assumption that $|\Im\nu|$ is sufficiently large and the proof is complete.
\end{proof}

The previous results allow us to conclude that \eqref{eq:PlancherelWithAandQ} is valid for all $\mu\in\CC$ with $\Re\mu<\frac{n}{2}$.

\subsection{Branching laws}

We now specialize \eqref{eq:PlancherelWithAandQ} to those parameters $\lambda$ for which the left hand side is positive semidefinite and use it to read off the explicit branching laws for the corresponding irreducible unitary representations of $G$ restricted to $H$. We first note that it is sufficient to integrate over $i\RR_+$ because of the following:

\begin{lemma}\label{lemma:QisEven}
The function $\nu\mapsto Q(\delta,\mu,\nu)$ is even. 
\end{lemma}
\begin{proof}
Using the functional equation \eqref{eq:Functionaleq: TB=A} write 
\[
\mathbf{A}_{\xi,(\mu,-\mu)}^{\eta,-\nu}=\Gamma(\tfrac{\mu-\nu+\frac{n}{2}+\delta}{2})\Gamma(\tfrac{\nu-\mu+\frac{n}{2}+\delta}{2})\mathbf{S}_{(\xi_1,\eta,\xi_1),(\mu,-\nu,-\mu)}\circ \mathbf{B}_{\xi,(\mu,-\mu)}^{\eta,-\nu}.
\]
So with \eqref{eq:TonySymmetricH} we get that
\begin{align*}
\frac{Q(\delta,\mu,\nu)}{\Gamma(\frac{\mu-\nu+\frac{n}{2}+\delta}{2})\Gamma(\frac{\nu-\mu+\frac{n}{2}+\delta}{2})}&=\big (\mathbf{B}_{\xi,(\mu,-\mu)}^{\eta,\nu}f\,\big|\,\mathbf{S}_{(\xi_1,\eta,\xi_1),(\mu,-\nu,-\mu)}\circ \mathbf{B}_{\xi,(\mu,-\mu)}^{\eta,-\nu}f\big)_H\\
&=\big (\mathbf{S}_{(\xi_1,\eta,\xi_1),(\mu,\nu,-\mu)}\circ\mathbf{B}_{\xi,(\mu,-\mu)}^{\eta,\nu}f\,\big|\, \mathbf{B}_{\xi,(\mu,-\mu)}^{\eta,-\nu}f\big)_H,
\end{align*}
the result then follows from $(\,\cdot \,|\,\cdot \,)_H$ being symmetric.
\end{proof}

Applying the lemma to \eqref{eq:PlancherelWithAandQ} and using \eqref{eq:Functionaleq: TB=A} we obtain for all $\xi_1=\xi_2\in\ZZ/2\ZZ$ and $\lambda_1=-\lambda_2\in\{\mu\in\CC:\Re\mu<\frac{n}{2}\}$:
\begin{multline*}
    ( f\,|\, \mathbf{T}_{\xi,\lambda}f')_G = \frac{1}{2\sqrt{\pi}\Gamma(\frac{\lambda_2-\lambda_1+n}{2})}\sum_{\eta=0}^1\int_{i\RR_+} n_\mathbf{B}(\xi,\lambda,\eta,\nu)^2n_\mathbf{B}(\xi,-\lambda,\eta,-\nu)\\
    \times\big(\mathbf{B}_{\xi,\lambda}^{\eta,\nu}f\,\big|\,\mathbf{S}_{(\xi_1,\eta,\xi_2),(\lambda_1,-\nu,\lambda_2)}\circ\mathbf{B}_{\xi,\lambda}^{\eta,-\nu}f'\big)_H\,d\nu.
\end{multline*}
For $\lambda_1=-\lambda_2\in(-\infty,\frac{n}{2})$ and $f'=\overline{f}$ with $f\in\pi_{\xi,\lambda}$, the left hand side becomes the invariant Hermitian form of Theorem~\ref{thm:BigHermitianFormPositiveSemidefinite} and the right hand side becomes an integral over the invariant Hermitian form of Theorem~\ref{thm:SmallInvariantHermitianFormPositiveSemidefinite} evaluated at $\mathbf{B}_{\xi,\lambda}^{\eta,\nu}f$:
\begin{equation}
    \langle f,f\rangle_{\xi,\lambda} = \frac{1}{2\sqrt{\pi}\Gamma(\frac{\lambda_2-\lambda_1+n}{2})}\sum_{\eta=0}^1\int_{i\RR_+} n_\mathbf{B}(\xi,\lambda,\eta,\nu)^2n_\mathbf{B}(\xi,-\lambda,\eta,-\nu)\big\langle \mathbf{B}_{\xi,\lambda}^{\eta,\nu}f,\mathbf{B}_{\xi,\lambda}^{\eta,\nu}f\big\rangle _{\eta,\nu}\,d\nu.\label{eq:FinalPlancherelFormula}
\end{equation}
Note that $n_\mathbf{B}(\xi,\lambda,\eta,\nu)^2n_\mathbf{B}(\xi,-\lambda,\eta,-\nu)>0$ for all $\lambda\in\RR^2$ with $\lambda_1+\lambda_2=0$ and $\nu\in i\RR_+$. Moreover, by Proposition~\ref{prop:B-surjective}, the operator $\mathbf{B}_{\xi,\lambda}^{\eta,\nu}$ is surjective under the same assumptions. This immediately implies the first part of the following statement:

\begin{theorem}\label{thm:BranchingLaws}
    For $\xi_1=\xi_2$ and
    $$ \lambda_1=-\lambda_2\in\{\ldots,-\tfrac{5}{2},-\tfrac{3}{2},-\tfrac{1}{2}\}\cup(-\tfrac{1}{2},\tfrac{1}{2})\cup\{\tfrac{1}{2},1,\ldots,\tfrac{n-1}{2}\} $$ we have the following branching law for the unitary completion of $\pi_{\xi,\lambda}/\ker(\mathbf{T}_{\xi,\lambda})$:
    \begin{align*}
        \big[\pi_{\xi,\lambda}/\ker(\mathbf{T}_{\xi,\lambda})\big]\Big|_H &\simeq \bigoplus_{\eta\in\ZZ/2\ZZ}\int_{i\RR_+}^\oplus\big(\tau_{(\xi_1,\eta,\xi_2),(\lambda_1,\nu,\lambda_2)}/\ker(\mathbf{S}_{(\xi_1,\eta,\xi_2),(\lambda_1,\nu,\lambda_2)})\big)\,d\nu\\
        &\simeq \bigoplus_{\eta\in\ZZ/2\ZZ}\int_{i\RR_+}^\oplus\Ind_{Q_H}^H\Big(\big[\varpi_{\xi,\lambda}/\ker(\mathbf{R}_{\xi,\lambda})\big]\otimes\chi_{\eta,\nu}\Big)\,d\nu.
    \end{align*}
\end{theorem}

Here the second isomorphism uses the degenerate series $\varpi_{\xi,\lambda}$ of $\GL(2n-2,\RR)$ and their intertwining operators $\mathbf{R}_{\xi,\lambda}$ (see Section~\ref{sec:UnirrepsH} for details).

\begin{proof}
    It only remains to show the second isomorphism, but this is precisely content of Theorem~\ref{thm:SmallInvariantHermitianFormPositiveSemidefinite}.
\end{proof}

Using the notation from Sections~\ref{sec:UnirrepsG} and \ref{sec:UnirrepsH} we can make the branching laws explicit for the four different types of representations:

\begin{corollary}\label{cor:BranchingLaws}
    For $\xi_1=\xi_2$ the following unitary branching laws hold:
    \begin{enumerate}
        \item\label{cor:BranchingLaws1} (Unitary degenerate series) For $\lambda_1=-\lambda_2\in i\RR$:
        $$ \pi_{\xi,\lambda}|_H \simeq \bigoplus_{\eta\in\ZZ/2\ZZ}\int_{i\RR_+}^\oplus\tau_{(\xi_1,\eta,\xi_2),(\lambda_1,\nu,\lambda_2)}\,d\nu \simeq \bigoplus_{\eta\in\ZZ/2\ZZ}\int_{i\RR_+}^\oplus\Ind_{Q_H}^H(\varpi_{\xi,\lambda}\otimes\chi_{\eta,\nu})\,d\nu. $$
        \item\label{cor:BranchingLaws2} (Stein's complementary series) For $\lambda_1=-\lambda_2\in(-\tfrac{1}{2},0)$:
        $$ \pi_{\xi,\lambda}^{\mathrm{c.s.}}|_H \simeq \bigoplus_{\eta\in\ZZ/2\ZZ}\int_{i\RR_+}^\oplus\tau_{(\xi_1,\eta,\xi_2),(\lambda_1,\nu,\lambda_2)}^{\mathrm{c.s.}}\,d\nu \simeq \bigoplus_{\eta\in\ZZ/2\ZZ}\int_{i\RR_+}^\oplus\Ind_{Q_H}^H(\varpi_{\xi,\lambda}^{\mathrm{c.s.}}\otimes\chi_{\eta,\nu})\,d\nu. $$
        \item\label{cor:BranchingLaws3} (Small irreducible quotients) For $\lambda_1=-\lambda_2\in\{\frac{1}{2},1,\ldots,\frac{n-1}{2}\}$:
        $$ \pi_{\xi,\lambda}^{\mathrm{small}}|_H \simeq \bigoplus_{\eta\in\ZZ/2\ZZ}\int_{i\RR_+}^\oplus\tau_{(\xi_1,\eta,\xi_2),(\lambda_1,\nu,\lambda_2)}^{\mathrm{small}}\,d\nu \simeq \bigoplus_{\eta\in\ZZ/2\ZZ}\int_{i\RR_+}^\oplus\Ind_{Q_H}^H(\varpi_{\xi,\lambda}^{\mathrm{small}}\otimes\chi_{\eta,\nu})\,d\nu. $$
        \item\label{cor:BranchingLaws4} (Speh representations) For $\lambda_1=-\lambda_2\in\{-\frac{1}{2},-\frac{3}{2},-\frac{5}{2},\ldots\}$:
        $$ \pi_{\xi,\lambda}^{\mathrm{Speh}}|_H \simeq \bigoplus_{\eta\in\ZZ/2\ZZ}\int_{i\RR_+}^\oplus\tau_{(\xi_1,\eta,\xi_2),(\lambda_1,\nu,\lambda_2)}^{\mathrm{large}}\,d\nu \simeq \bigoplus_{\eta\in\ZZ/2\ZZ}\int_{i\RR_+}^\oplus\Ind_{Q_H}^H(\varpi_{\xi,\lambda}^{\mathrm{Speh}}\otimes\chi_{\eta,\nu})\,d\nu. $$
    \end{enumerate}
\end{corollary}

\begin{remark}
    In view of Figure~\ref{fig:UnirrepsG} and \ref{fig:UnirrepsH} the branching laws can be visualized as follows: For each dot in Figure~\ref{fig:UnirrepsG} the restriction to $H$ of the corresponding irreducible unitary representation of $G$ decomposes into a direct integral of the irreducible unitary representations of $H$ in Figure~\ref{fig:UnirrepsH} that correspond to the line at the same position.
\end{remark}

\subsection{Relation to other work}

We comment on several relations of our results to the existing literature.

\begin{remark}
    Note that all representations occurring in the decomposition of a complementary series representation are themselves complementary series representations. The authors are not aware of such a phenomenon in the existing literature.
\end{remark}

\begin{remark}\label{rem:Adduced}
    The decomposition of complementary series and Speh representations can also be deduced from the theory of adduced representations. In \cite[Theorem 3]{SS90} and \cite[Theorem 1~(c)]{Sah90} it is shown that the representations $\pi_{\xi,\lambda}^{\mathrm{c.s.}}$ and $\pi_{\xi,\lambda}^{\mathrm{Speh}}$ are adducable of depth $2$ and their adduced representations of $\GL(2n-2,\RR)$ are $\varpi_{\xi,\lambda}^{\mathrm{c.s.}}$ and $\varpi_{\xi,\lambda}^{\mathrm{Speh}}$. That a representation $\varpi$ of $\GL(2n-2,\RR)$ is the adduced representation of a representation $\pi$ of $\GL(2n,\RR)$ means that
    $$ \pi|_{M_{2n}} \simeq \Ind_{M_{2n-1}\ltimes\RR^{2n-1}}^{M_{2n}}(\sigma\otimes\chi), $$
    where $M_k\subseteq G_k$ denotes the mirabolic subgroup whose last row equals $(0,\ldots,0,1)$, $\sigma$ is extended trivially to a representation of $M_{2n-1}\simeq G_{2n-2}\ltimes\RR^{2n-2}$ and $\chi$ is a certain character of $\RR^{2n-1}$. Since $M_{2n}/(M_{2n-1}\ltimes\RR^{2n-1})\simeq G_{2n-1}/M_{2n-1}$, it follows that
    $$ \pi|_{G_{2n-1}} \simeq \Ind_{M_{2n-1}}^{G_{2n-1}}(\sigma), $$
    and by induction in stages we find that the latter representation is isomorphic to
    $$ \Ind^H_{Q_H}\big(\Ind_{M_{2n-1}}^{Q_H}(\sigma)\big). $$
    But $Q_H/M_{2n-1}\simeq\RR^\times$ as groups, so
    $$ \Ind_{M_{2n-1}}^{Q_H}(\sigma) \simeq \big(\sigma\otimes L^2(\RR^\times)\big)\otimes1 $$
    as representation of $Q_H\simeq(\GL(2n-2,\RR)\times\GL(1,\RR))\ltimes\RR^{2n-2}$. Decomposing $L^2(\RR^\times)\simeq\bigoplus_{\eta\in\ZZ/2\ZZ}\int^\oplus_{i\RR_+}\chi_{\eta,\nu}\,d\nu$ as representation of $\GL(1,\RR)=\RR^\times$ shows that
    $$ \pi|_H \simeq \bigoplus_{\eta\in\ZZ/2\ZZ}\int^\oplus_{i\RR_+} \Ind_{Q_H}^H(\sigma\otimes\chi_{\eta,\nu})\,d\nu. $$
\end{remark}

\begin{remark}
    There is an easier way to obtain the branching law for the small irreducible quotients using the observation of \cite{BSS90} that $\pi_{\xi,\lambda}^{\mathrm{small}}$ for $\lambda_1=-\lambda_2=\frac{n-k}{2}$ is in fact isomorphic to a unitary degenerate series representation induced from a character of the parabolic subgroup of $G$ with Levi factor $\GL(k,\RR)\times\GL(2n-k,\RR)$. The decomposition of the restriction of this unitary degenerate series to $H$ is a simple application of Mackey theory in the same spirit as is carried out in Section~\ref{sec:Plancherel}. The open $H$-orbit in the corresponding Grassmannian is an $\RR^\times$-bundle over $H/P_{H,k}$ with $P_{H,k}$ a parabolic subgroup of $H$ with Levi factor $\GL(k-1,\RR)\times\GL(1,\RR)\times\GL(2n-k-1,\RR)$. Therefore, the decomposition is obtained by a Mellin transform along the fiber $\RR^\times$ and consists of a direct integral of all unitary degenerate series representation induced from $P_{H,k}$. Writing the unitarizable quotient $\varpi_{\xi,\lambda}^{\mathrm{small}}$ in the same way as a unitary degenerate series induced from a parabolic subgroup of $\GL(2n-2,\RR)$ with Levi factor $\GL(k-1,\RR)\times\GL(2n-k-1,\RR)$ and using induction in stages, we recover the branching law in Corollary~\ref{cor:BranchingLaws}~\eqref{cor:BranchingLaws3}.
\end{remark}

\begin{remark}\label{rem:KobayashiSpeh}
    In \cite[Section 4.2]{KS25} a family of representations is introduced that appears to be very similar to the large irreducible quotients $\tau_{(\xi_1,\eta,\xi_2),(\lambda_1,\nu,\lambda_2)}^{\mathrm{large}}$ (which are isomorphic to $\Ind_{Q_H}^H(\varpi_{\xi,\lambda}^{\mathrm{Speh}}\otimes\chi_{\eta,\nu})$ by Theorem~\ref{thm:SmallInvariantHermitianFormPositiveSemidefinite}). We believe that these two families of representations are in fact isomorphic. The representations in \cite[Section 4.2]{KS25} are constructed using cohomological induction from the parabolic subgroup $Q$ of $H_\CC=\GL(2n-1,\CC)$ with Levi factor $L_Q=\GL(n-1,\CC)\times\GL(1,\CC)\times\GL(n-1,\CC)$. If the induction parameter is the character of $L_Q$ given by
    $$ L_Q\to\CC^\times, \quad (l_1,l_2,l_3)\mapsto |\det(l_1)|_{\xi_1}^{\lambda_1}|l_2|_\eta^\nu|\det(l_3)|_{\xi_2}^{\lambda_3}, $$
    then the cohomologically induced representation is unitarizable and has the same infinitesimal character as $\tau_{(\xi_1,\eta,\xi_2),(\lambda_1,\nu,\lambda_2)}$, which makes us believe that it is isomorphic to the quotient $\tau_{(\xi_1,\eta,\xi_2),(\lambda_1,\nu,\lambda_2)}^{\mathrm{large}}$.\\
    Moreover, a special case of \cite[Theorem 4.1]{KS25} implies the existence of a non-trivial $H$-intertwining operator between the smooth vectors of $\pi_{\xi,\lambda}^{\mathrm{Speh}}$ and $\tau_{(\xi_1,\eta,\xi_2),(\lambda_1,\nu,\lambda_2)}^{\mathrm{large}}$ for $\eta=0$. Our analysis shows that $\mathbf{B}_{\xi,\lambda}^{\eta,\nu}$ (or one of its renormalizations) is such a non-zero operator for every $\eta\in i\RR$.
\end{remark}

\appendix

\section{Proof of Theorem~\ref{thm:DecayOfBinNu}}\label{app:ProofOfEstimate}

The statement will be proved in five steps.

\subsection*{Step 1}

We first note that it suffices to show the inequality
\begin{equation}
    \left|B_{\xi,\lambda}^{\eta,\nu}f(e)\right| \leq q(f)(1+|\Im\nu|)^{-N}.\label{eq:DecayOfBatEinNu}
\end{equation}
So let us assume that \eqref{eq:DecayOfBatEinNu} holds. Using the intertwining property of $B_{\xi,\lambda}^{\eta,\nu}$, we find for every $h\in H$:
$$ B_{\xi,\lambda}^{\eta,\nu}f(h) = \tau_{\eta,\nu}(h)^{-1}\circ B_{\xi,\lambda}^{\eta,\nu}f(e) = B_{\xi,\lambda}^{\eta,\nu}\circ\pi_{\xi,\lambda}(h)^{-1}f(e). $$
By \eqref{eq:DecayOfBatEinNu}, the absolute value of the right hand side can be estimated by
$$ q(\pi_{\xi,\lambda}(h)^{-1}f)(1+|\Im\nu|)^{-N}, $$
and since $q$ is a continuous seminorm and $\pi_{\xi,\lambda}$ is continuous, the supremum of this expression over all $h\in K_{2n-1}$ is another continuous seminorm applied to $f$ times $(1+|\Im\nu|)^{-N}$.

\subsection*{Step 2}
 
Next, we prove two Bernstein--Sato identities for $B_{\xi,\lambda}^{\eta,\nu}$ in the compact picture. For this we fix $f\in C^\infty(K_{2n}\times_{M_G\cap K_{2n}}\xi)$ and view it as a vector in $\pi_{\xi,\lambda}$ for every $\lambda$. By \eqref{eq:BinKcoordinates} we can write
$$ \quad\qquad B_{\xi,\lambda}^{\eta,\nu}f(h) = \int_{-\frac{\pi}{2}}^{\frac{\pi}{2}}u_{\xi,\lambda}^{\eta,\nu}(\theta)f(hk_\theta)\,d\theta, \quad \mbox{where} \quad u_{\xi,\lambda}^{\eta,\nu}(\theta)=|\sin\theta|^{\lambda_1-\nu+\frac{n}{2}-1}_{\xi_1+\eta}(\cos\theta)^{\nu-\lambda_2+\frac{n}{2}-1}. $$
It is straightforward to verify that $u_{\xi,\lambda}^{\eta,\nu}(\theta)$ satisfies the two Bernstein--Sato identities
\begin{align*}
    \left(\cos\theta\frac{d}{d\theta}+(\lambda_1-\lambda_2+n-2)\sin\theta\right)u_{\xi,\lambda}^{\eta,\nu} &= (\lambda_1-\nu+\tfrac{n}{2}-1)u_{\xi-e_1,\lambda-e_1}^{\eta,\nu},\\
    \left(-\sin\theta\frac{d}{d\theta}+(\lambda_1-\lambda_2+n-2)\cos\theta\right)u_{\xi,\lambda}^{\eta,\nu} &= (\nu-\lambda_2+\tfrac{n}{2}-1)u_{\xi+e_2,\lambda+e_2}^{\eta,\nu}.
\end{align*}
Integrating by parts shows that
\begin{multline*}
    \qquad\qquad(\lambda_1-\nu+\tfrac{n}{2}-1)B_{\xi-e_1,\lambda-e_1}^{\eta,\nu}f(h)\\
    = \int_{-\frac{\pi}{2}}^{\frac{\pi}{2}}u_{\xi,\lambda}^{\eta,\nu}(\theta)\left(-\cos\theta\frac{d}{d\theta}+(\lambda_1-\lambda_2+n-1)\sin\theta\right)f(hk_\theta)\,d\theta
\end{multline*}
and similar for the other identity. Since $\cos\theta=(hk_\theta)_{2n,2n}$, $\sin\theta=-(k_\theta)_{2n,n}$ and $\frac{d}{d\theta}hk_\theta=\sum_{i=1}^{2n-1}(\cos\theta E_{i,2n}-\sin\theta E_{i,n})h_{i,n}-\cos\theta E_{2n,n}-\sin\theta E_{n,2n}$, this can be rewritten as
\begin{align*}
    (B_{\xi,\lambda}^{\eta,\nu}\circ D_1)f(h) &= (\lambda_1-\nu+\tfrac{n}{2}-1)B_{\xi-e_1,\lambda-e_1}^{\eta,\nu}f(h),\\
    (B_{\xi,\lambda}^{\eta,\nu}\circ D_2)f(h) &= (\nu-\lambda_2+\tfrac{n}{2}-1)B_{\xi+e_2,\lambda+e_2}^{\eta,\nu}f(h),
\end{align*}
where
\begin{align*}
    D_1 &= -g_{2n,2n}dr(E_{2n,n})-(\lambda_1-\lambda_2+n-1)g_{2n,n},\\
    D_2 &= g_{2n,n}dr(E_{2n,n})+(\lambda_1-\lambda_2+n-1)g_{2n,2n},
\end{align*}
$dr$ denoting the derived representation of the right regular action of $K_{2n}$ on $C^\infty(K_{2n}\times_{M_G\cap K_{2n}}\xi)$.

\subsection*{Step 3}

Next, we reduce \eqref{eq:DecayOfBatEinNu} to the case where $\lambda_1-R-\frac{n}{2},R-\lambda_2-\frac{n}{2}\geq0$. Iterating the Bernstein--Sato identities from Step 2, we obtain for $a,b\in\NN$:
\begin{equation}\label{eq:thm:DecayOfAinNu1}
    B_{\xi,\lambda}^{\eta,\nu}f = b(\lambda,\nu)^{-1}B_{\xi,\lambda+2ae_1-2be_2}^{\eta,\nu}\circ D_1^{2a}D_2^{2b}f
    \end{equation}
for some polynomial $b(\lambda,\nu)$. Now choose $a$ and $b$ such that $\lambda_1+2a-R-\frac{n}{2},R+2b-\lambda_2-\frac{n}{2}\geq0$. Then $\lambda+2ae_1-2be_2$ satisfies the inequalities above, so if we assume that \eqref{eq:DecayOfBatEinNu} holds in this case, then we obtain for every fixed $N\in\NN$ a seminorm $q$ on $\pi_{\xi,\lambda+2ae_1-2be_2}$ such that
\begin{equation}
    \left|B_{\xi,\lambda+2ae_1-2be_2}^{\eta,\nu}g(e)\right| \leq q(g)(1+|\Im\nu|)^{-N}\label{eq:thm:DecayOfAinNu2}
\end{equation}
for all $g\in\pi_{\xi,\lambda+2ae_1-2be_2}$ whenever $\nu\in\CC$ with $|\Re\nu|\leq R$. Combining \eqref{eq:thm:DecayOfAinNu1} and \eqref{eq:thm:DecayOfAinNu2} for $g=D_1^{2a}D_2^{2b}f$ we obtain
$$ \left|B_{\xi,\lambda}^{\eta,\nu}f(e)\right| \leq |b(\lambda,\nu)|^{-1}q(D_1^{2a}D_2^{2b}f)(1+|\Im\nu|)^{-N}. $$
Now, $f\mapsto q(D_1^{2a}D_2^{2b}f)$ defines a continuous seminorm on $\pi_{\xi,\lambda}$, and $b(\lambda,\nu)^{-1}$ is for $|\Re\nu|\leq R$ and $|\Im\nu|\gg0$ bounded, so we have proved $\eqref{eq:DecayOfBatEinNu}$ for arbitrary $\lambda$.

\subsection*{Step 4}

By the previous step we may assume that $\lambda_1-R-\frac{n}{2},R-\lambda_2-\frac{n}{2}\geq0$. We now first show \eqref{eq:DecayOfBatEinNu} for $N=0$. Since $|\Re\nu|\leq R$ we find that $\Re(\lambda_1-\nu-\frac{n}{2}),\Re(\nu-\lambda_2-\frac{n}{2})\geq0$. Therefore, the kernel $u_{\xi,\lambda}^{\eta,\nu}$ of $B_{\xi,\lambda}^{\eta,\nu}$ is bounded by a constant $C$ for all $\nu\in\CC$ with $|\Re\nu|\leq R$:
$$ \left|u_{\xi,\lambda}^{\eta,\nu}(\theta)\right| \leq C \qquad \mbox{for all }\theta\in\RR,\nu\in\CC,|\Re\nu|\leq R. $$
Since
$$ B_{\xi,\lambda}^{\eta,\nu}f(e) = \int_{-\frac{\pi}{2}}^{\frac{\pi}{2}}u_{\xi,\lambda}^{\eta,\nu}(\theta)f(k_\theta)\,d\theta, $$
it follows that
$$ \left|B_{\xi,\lambda}^{\eta,\nu}f(e)\right| \leq C\|f\|_{L^\infty(K_{2n})} \qquad \mbox{for all }f\in\pi_{\xi,\lambda}. $$
Since the $L^\infty$-norm is a continuous seminorm on $\pi_{\xi,\lambda}$, this shows the claim for $N=0$.

\subsection*{Step 5}

The last step is to show \eqref{eq:DecayOfBatEinNu} for $\lambda_1-R-\frac{n}{2},R-\lambda_2-\frac{n}{2}\geq0$ and arbitrary $N\in\NN$ by induction on $N$. The induction start was done in the previous step. For the induction step we use again one of the Bernstein--Sato identities, namely
$$ B_{\xi+e_1,\lambda+e_1}^{\eta,\nu}\circ D_1 = (\lambda_1-\nu-\tfrac{n}{2})B_{\xi,\lambda}^{\eta,\nu}. $$
The left hand side applied to $f$ and evaluated at $e$ can be written as
\begin{align*}
    B_{\xi+e_1,\lambda+e_1}^{\eta,\nu}\circ D_1f(e) &= \int_{-\frac{\pi}{2}}^{\frac{\pi}{2}}|\sin\theta|^{\lambda_1-\nu+\frac{n}{2}}_{\xi_1+\eta+1}(\cos\theta)^{\nu-\lambda_2+\frac{n}{2}-1}\cdot D_1f(k_\theta)\,dk\\
    &= \int_{-\frac{\pi}{2}}^{\frac{\pi}{2}}|\sin\theta|^{\lambda_1-\nu+\frac{n}{2}-1}_{\xi_1+\eta}(\cos\theta)^{\nu-\lambda_2+\frac{n}{2}-1}\cdot \sin\theta D_1f(k)\,dk\\
    &= B_{\xi,\lambda}^{\eta,\nu}\circ g_{2n,n}D_1f(e).
\end{align*}
If now
$$ \left|B_{\xi,\lambda}^{\eta,\nu}f(e)\right| \leq q(f)(1+|\Im\nu|)^{-N}, $$
then
\begin{align*}
    \left|B_{\xi,\lambda}^{\eta,\nu}f(e)\right| &= \frac{1}{|\lambda_1-\nu-\tfrac{n}{2}|}\left|B_{\xi,\lambda}^{\eta,\nu}\circ g_{2n,n}D_1f(e)\right|\\
    &\leq \frac{(1+|\Im\nu|)^{-N}}{|\lambda_1-\nu-\tfrac{n}{2}|}q(g_{2n,n}D_1f).
\end{align*}
For fixed $\lambda$ and bounded $\Re\nu$, we can asymptotically replace $|\lambda_1-\nu-\frac{n}{2}|$ by $(1+|\Im\nu|)$ up to a constant to improve the estimate from $N$ to $N+1$.\qed

\bibliographystyle{amsplain}
\bibliography{bibdb}

\providecommand{\bysame}{\leavevmode\hbox to3em{\hrulefill}\thinspace}
\providecommand{\MR}{\relax\ifhmode\unskip\space\fi MR }
\providecommand{\MRhref}[2]{%
  \href{http://www.ams.org/mathscinet-getitem?mr=#1}{#2}
}
\providecommand{\href}[2]{#2}
\begin{thebibliography}{10}

\bibitem{BSS90}
Dan Barbasch, Siddhartha Sahi, and Birgit Speh, \emph{Degenerate series representations for {${\rm GL}(2n,{\bf R})$} and {F}ourier analysis}, Symposia {M}athematica, {V}ol.\ {XXXI} ({R}ome, 1988), Sympos. Math., vol. XXXI, Academic Press, London, 1990, pp.~45--69.

\bibitem{BSZ06}
Leticia Barchini, Mark Sepanski, and Roger Zierau, \emph{Positivity of zeta distributions and small unitary representations}, The ubiquitous heat kernel, Contemp. Math., vol. 398, Amer. Math. Soc., Providence, RI, 2006, pp.~1--46.

\bibitem{CF23}
Corina Ciobotaru and Jan Frahm, \emph{Symmetry breaking for $\mathrm{PGL}(2)$ over non-archimedean local fields}, to appear in IMS Lecture Note Series for Representations and Characters: Revisiting the Works of Harish-Chandra and Andr\'{e} Weil, available at \href{https://arxiv.org/abs/2309.14864}{arXiv:2309.14864}, 2023.

\bibitem{DF24}
Jonathan Ditlevsen and Jan Frahm, \emph{Construction and analysis of symmetry breaking operators for the pair $(\mathrm{GL}(n+1,\mathbb{R}),\mathrm{GL}(n,\mathbb{R}))$}, preprint, available at \href{https://arxiv.org/abs/2403.14267}{arXiv:2403.14267}, 2024.

\bibitem{DL25}
Jonathan Ditlevsen and Quentin Labriet, \emph{Differential symmetry breaking operators for the pair $(\mathrm{GL}_{n+1}(\mathbb{R}),\mathrm{GL}_n(\mathbb{R}))$}, preprint, available at \href{https://arxiv.org/abs/2504.20793}{arXiv:2504.20793}, 2025.

\bibitem{FW}
Jan Frahm and Clemens Weiske, \emph{Branching laws for unitary representations of {${\rm U}( 1, n + 1 )$} in the scalar principal series}, in preparation.

\bibitem{GS64}
Israel~M. Gel'fand and Georgi~E. Shilov, \emph{Generalized functions}, Academic Press, New York and London, 1964.

\bibitem{K86}
Anthony Knapp, \emph{Representation theory of semisimple groups}, Princeton Math. Ser., vol. 36, Princeton University Press, 1986.

\bibitem{Kob15}
Toshiyuki Kobayashi, \emph{A program for branching problems in the representation theory of real reductive groups}, Representations of reductive groups, Progr. Math., vol. 312, Birkh\"auser/Springer, Cham, 2015, pp.~277--322.

\bibitem{TOP11}
Toshiyuki Kobayashi, Bent {\O}rsted, and Michael Pevzner, \emph{Geometric analysis on small unitary representations of {$\mathrm{GL}(N,\mathbb{R})$}}, Journal of Functional Analysis \textbf{260} (2011), no.~6, 1682--1720.

\bibitem{KP16}
Toshiyuki Kobayashi and Michael Pevzner, \emph{Differential symmetry breaking operators: {I}. {G}eneral theory and {F}-method}, Selecta Math. (N.S.) \textbf{22} (2016), no.~2, 801--845.

\bibitem{KS15}
Toshiyuki Kobayashi and Birgit Speh, \emph{Symmetry breaking for representations of rank one orthogonal groups}, Mem. Amer. Math. Soc. \textbf{238} (2015), no.~1126.

\bibitem{KS18}
\bysame, \emph{Symmetry breaking for representations of rank one orthogonal groups {II}}, Lecture Notes in Mathematics, vol. 2234, Springer, Singapore, 2018.

\bibitem{KS25}
\bysame, \emph{How does the restriction of representations change under translations? {A} story for the general linear groups and the unitary groups}, preprint, available at \href{https://arxiv.org/abs/2502.08479}{arXiv:2502.08479}, 2025.

\bibitem{MOZ16}
Jan M\"{o}llers, Bent {\O}rsted, and Genkai Zhang, \emph{Invariant differential operators on {H}-type groups and discrete components in restrictions of complementary series of rank one semisimple groups}, J. Geom. Anal. \textbf{26} (2016), no.~1, 118--142.

\bibitem{MO15}
Jan M\"ollers and Yoshiki Oshima, \emph{Restriction of most degenerate representations of {$O(1,N)$} with respect to symmetric pairs}, J. Math. Sci. Univ. Tokyo \textbf{22} (2015), no.~1, 279--338.

\bibitem{Sah89}
Siddhartha Sahi, \emph{On {K}irillov's conjecture for {A}rchimedean fields}, Compositio Math. \textbf{72} (1989), no.~1, 67--86.

\bibitem{Sah90}
\bysame, \emph{A simple construction of {S}tein's complementary series representations}, Proc. Amer. Math. Soc. \textbf{108} (1990), no.~1, 257--266.

\bibitem{Sah95}
\bysame, \emph{Jordan algebras and degenerate principal series}, J. Reine Angew. Math. \textbf{462} (1995), 1--18.

\bibitem{SS90}
Siddhartha Sahi and Elias~M. Stein, \emph{Analysis in matrix space and {S}peh's representations}, Invent. Math. \textbf{101} (1990), no.~2, 379--393.

\bibitem{Spe77}
Birgit Speh, \emph{Some results on principal series of {$\mathrm{GL}(n,\mathbb{R})$}}, Ph.D. thesis, Massachusetts Institute of Technology, 1977.

\bibitem{Spe81}
\bysame, \emph{The unitary dual of {${\rm Gl}(3,\,{\bf R})$}\ and {${\rm Gl}(4,\,{\bf R})$}}, Math. Ann. \textbf{258} (1981/82), no.~2, 113--133.

\bibitem{SV11}
Birgit Speh and T.~N. Venkataramana, \emph{Discrete components of some complementary series}, Forum Math. \textbf{23} (2011), no.~6, 1159--1187.

\bibitem{SV12}
\bysame, \emph{On the restriction of representations of {${\rm SL}(2,\Bbb C)$} to {${\rm SL}(2,\Bbb R)$}}, Representation theory, complex analysis, and integral geometry, Birkh\"auser/Springer, New York, 2012, pp.~231--249.

\bibitem{SZ16}
Birgit Speh and Genkai Zhang, \emph{Restriction to symmetric subgroups of unitary representations of rank one semisimple {L}ie groups}, Math. Z. \textbf{283} (2016), no.~1-2, 629--647.

\bibitem{Ste67}
Elias~M. Stein, \emph{Analysis in matrix spaces and some new representations of {${\rm SL}(N,\,\mathbf{C})$}}, Ann. of Math. (2) \textbf{86} (1967), 461--490.

\bibitem{vDM99}
Gerrit van Dijk and Vladimir~F. Molchanov, \emph{Tensor products of maximal degenerate series representations of the group {${\mathrm SL} (n,\mathbf R)$}}, J. Math. Pures Appl. (9) \textbf{78} (1999), no.~1, 99--119.

\bibitem{Vog86}
David~A. Vogan, Jr., \emph{The unitary dual of {${\rm GL}(n)$} over an {A}rchimedean field}, Invent. Math. \textbf{83} (1986), no.~3, 449--505.

\bibitem{Wei24}
Clemens Weiske, \emph{Branching of unitary {{${\rm O}(1,n+1)$}}-representations with non-trivial {{$(\mathfrak{g}, K)$}}-cohomology}, Ann. Inst. Fourier (Grenoble) \textbf{74} (2024), no.~6, 2331--2377.

\end{thebibliography}

\end{document}